\documentclass[11pt]{article}
\usepackage[left=1in, right=1in, top=1in]{geometry}
\usepackage[toc,page]{appendix}
\usepackage[utf8]{inputenc}
\usepackage{amsfonts}
\usepackage{mathrsfs}
\usepackage[english]{babel}
\usepackage{amsmath}
\usepackage{relsize}
\usepackage{lipsum}
\usepackage{xcolor}
\usepackage{amsthm}
\usepackage{tikz-cd}
\usepackage{extpfeil}
\usepackage{mathtools}
\usepackage{csquotes}
\usepackage{soul} 
\usepackage{amssymb}
\usepackage[normalem]{ulem}
\usepackage[nottoc,numbib]{tocbibind} 
\usepackage{graphicx}
\usepackage{authblk}
\usepackage{amscd} 
\newcommand{\Phisig}{\Phi_{\sigma}}
\newcommand{\ksig}{K_{\sigma}}
\newcommand{\ballsig}{\ball_{\sigma}}
\newcommand{\qonepull}{q_1^*}
\usepackage[colorlinks=true,citecolor=blue,linkcolor=black,urlcolor=black]{hyperref}

\newcommand{\cc}{\mc{C}}
\newcommand{\id}{\on{id}}

\newcommand{\Aut}{\on{Aut}}

\newcommand{\pull}{{}^*}

\newcommand{\fgam}{f_{\gamma}}
\newcommand{\Fgam}{F_{\gamma}}

\newcommand{\ctop}{C^{\on{top}}}
\newcommand{\ufrxt}{\underline{\Fr}_{X/T}}

\newcommand{\pan}{P^{\on{an}}}
\newcommand{\ball}{\mathbb{B}}

\newcommand{\ponepull}{p_1^*}

\newcommand{\mce}{\mc{E}}

\newcommand{\pr}{\on{pr}}

\newcommand{\bracket}{[ \ , \ ]}
\newcommand{\hra}{\hookrightarrow}

\newcommand{\phipull}{\phi^*}

\newcommand{\mcz}{\mc{Z}}

\newcommand{\mcl}{\mc{L}}

\newcommand{\pibos}{\pi_{\bos}}
\newcommand{\pipull}{\pi^*}

\newcommand{\fpull}{f^*}

\newcommand{\lra}[1]{\overset{#1}{\longrightarrow}}

\newcommand{\bos}{\on{bos}}

\newcommand{\xbos}{X_{\on{bos}}}

\newcommand{\tbos}{T_{\on{bos}}}

\newcommand{\oc}{\mc{O}_C}
\newcommand{\jj}{\mc{J}}

\newcommand{\ot}{\mc{O}_T}

\newcommand{\ox}{\mc{O}_X}

\newcommand{\mc}[1]{\mathcal{#1}}
\newcommand{\on}[1]{\operatorname{#1}}
\newcommand{\ov}[1]{\overline{#1}}
\newcommand{\sS}{\on{sSch}}

\newcommand{\pist}{\pi_*} 

\newcommand{\lgr}{\overset{\sim}{\longrightarrow}}
\newcommand{\bb}{\mathbb{C}}
\newcommand{\pp}{\mathbb{P}}
\newcommand{\z}{\mathbb{Z}}
\newcommand{\rr}{\mathbb{R}}

\newcommand{\mfr}{\mathfrak{M}_g}

\newcommand{\mcx}{\mc{X}}

\newtheorem{theorem}{Theorem}[section]
\newtheorem{lemma}[theorem]{Lemma}
\newtheorem{prop}[theorem]{Proposition}
\newtheorem{corollary}[theorem]{Corollary}
\newtheorem{definition}[theorem]{Definition}

\newcommand{\tg}{\mathfrak{T}_g}

\newcommand{\inv}{{}^{-1}}

\newcommand{\embproj}{\pp^{5g-6|6g-6}}

\newcommand{\duaxtsix}{\omega_{X/T}^{6}}

\newcommand{\mct}{\mc{T}}

\newcommand{\xra}{\xrightarrow}

\newcommand{\Fr}{\on{Fr}}

\newcommand{\irho}{i_{\rho}}

\title{Super Teichm\"uller Space}

\author{Nadia Ott\thanks{Department of Mathematics, University of Southern Denmark, Odense, Denmark. Email: \href{mailto:nmott@imada.sdu.dk}{\texttt{nmott@imada.sdu.dk}}.}}
\date{}

\begin{document}
\maketitle

\begin{abstract}
 The paper contains three results. First, we give a global quotient presentation of the supermoduli space of genus $g>1$ super Riemann surfaces (Corollary~\ref{global quotient}). Second, we prove that the moduli stack of marked super Riemann surfaces, which we call \emph{super Teichm\"uller space} and denote by $\tg$, is a complex analytic Deligne–Mumford superstack (Proposition~\ref{super teich is DM}). Third, we construct an étale atlas for $\tg$ whose underlying topological space is a disjoint union of $2^{2g}$ contractible unit balls (Theorem~\ref{teich theorem for super}).

\end{abstract}

\section{Introduction}
Teichm\"uller space $T_g$ is the moduli space of Riemann surfaces with a \emph{marking}. A marking is the isotopy class of an orientation-preserving homeomorphism from the Riemann surface to a fixed surface $\Sigma$ of genus $g>1$. An isomorphism of marked Riemann surfaces preserves the marking. Adding a marking kills all non-trivial automorphisms of the curve but does not add new parameters to the moduli problem.

There is a natural forgetful map from Teichm\"uller space to the moduli space $\mc{M}_g$ of genus $g>1$ curves:
\begin{equation}
\label{introduction: forgetting the marking map}
\pi: T_g \to \mc{M}_g
\end{equation}
whose fibers are torsors for the mapping class group of $\Sigma$. Thus $\mc{M}_g$ is a quotient of $T_g$ by the mapping class group.
This is useful because $T_g$ has much simpler geometry than the moduli space: it is a complex manifold of dimension $3g-3$ whose underlying topological space is homeomorphic to a contractible unit ball.

There are many constructions and definitions of Teichm\"uller space. Examples include Teichm\"uller's quasiconformal construction \cite{teichmuller1940extremale}, Grothendieck's functor-of-points construction \cite{Grothendieck1961}, and the hyperbolic approach, where $T_g$ is identified with the space of conjugacy classes of discrete faithful representations $\pi_1(\Sigma)\to\on{PSL}(2,\rr)$ compatible with the orientation of $\Sigma$ \cite{goldman1988topological}. These constructions are connected by the classical uniformization theorem.

Super Teichm\"uller theory asks how the classical picture extends to super Riemann surfaces, where spin structures, odd deformation parameters, and automorphisms all enter the moduli problem. The subject arose in superstring theory in the 1980s. Early contributions include Rabin's work on deformations \cite{rabin1987teichmuller}, the uniformization approach of Crane and Rabin \cite{CR}, and the construction of supermoduli spaces by LeBrun and Rothstein \cite{LR}. Later analytic approaches to families include those of Sachse \cite{sachse2009global} and Knevel \cite{knevel2011versal}. Penner and Zeitlin develop a decorated theory for punctured surfaces using super hyperbolic geometry \cite{PZ}.

These approaches emphasize different aspects of the subject, and comparing them requires attention to the chosen markings, the treatment of automorphisms, and the relevant universal properties. Bryant and Hodgkin \cite{bryant1993nielsen} discuss how different notions of marking lead to different versions of super Teichm\"uller space. From the moduli-theoretic perspective, Codogni and Viviani \cite{CV} construct the complex analytic Deligne--Mumford superstack of super Riemann surfaces.

In this paper, we study super Teichm\"uller space $\tg$ as the moduli stack of marked families of super Riemann surfaces. Our aim is to give a construction using standard algebro-geometric language and methods, following Arbarello and Cornalba's construction of ordinary Teichm\"uller space \cite{AC}, in the spirit of Grothendieck's functor-of-points approach. We use moduli of framed super Riemann surfaces to construct Kuranishi families and an \'etale atlas for $\tg$, making its relationship with the supermoduli stack explicit.

We now describe the contents of the paper.

In Section~\ref{sec:frames}, we construct the ambient superscheme containing the bases of the Kuranishi families that we ultimately prove form an \'etale atlas for $\tg$. This ambient superscheme is the moduli space of \emph{framed super Riemann surfaces}. A framed super Riemann surface is a super Riemann surface equipped with a basis of global sections of its six-canonical line bundle $\omega_X^{\otimes 6}$, which provides an embedding into superprojective space $\embproj$.

This suggests that the six-canonical \emph{locus} $\mc{H}_{g,6}$ in the Hilbert superscheme $\mc{H}_g$ of closed subschemes of $\embproj$ \cite{bruzzo2023notes} may serve as a parameter space for super Riemann surfaces. However, the situation is more complicated: the six-canonical locus is not known to be representable. The basic obstruction is the presence of extra odd global sections along a certain divisor of super Riemann surfaces.
This presented the main technical difficulty of the paper, as the construction in \cite{AC} uses the tricanonical locus as the ambient space in which to define Kuranishi families.

The resolution was to work directly with the moduli space of framed super Riemann surfaces. We use the Hilbert superscheme and its superconformal enhancement to prove that this moduli space is represented by a smooth superscheme $P$ of finite type over $\bb$ (Proposition~\ref{prop:P-superscheme}).

We also provide a global quotient construction of supermoduli space (Corollary~\ref{global quotient}). As far as we are aware, the only global quotient presentation that currently exists in the literature is in genus $0$ with Ramond punctures \cite{ott2023supermoduli}.   Codogni and Viviani \cite{CV} explicitly raised the question of a construction using Hilbert superschemes of pluricanonically embedded super Riemann surfaces; our construction using framed surfaces follows this direction.

We then pass to the super analytic category in Section~\ref{section: super Teich space}. The analytifications of the framed moduli space $P$ and the projective linear supergroup are complex supermanifolds.
We show that every super Riemann surface of genus $g>1$ admits a Kuranishi family (Lemma~\ref{slice is kuranishi}). This result was already established in \cite{CV}, but we give a new proof using local slices in $P$, analogous to the construction in \cite{AC}.
In Section~\ref{sec: definition of STM}, we define marked super Riemann surfaces as super Riemann surfaces equipped with an isotopy class of orientation-preserving homeomorphisms $f: X \lgr \Sigma$. We define super Teichm\"uller space $\tg$ as the moduli groupoid of marked families.

Like marked spin curves, a marked super Riemann surface has exactly one non-trivial automorphism, the \emph{fermionic reflection}. Locally, it acts by sign change on the odd coordinate, $\theta \mapsto -\theta$. In the spin curve description, this is the automorphism $-\id_L$ of the spin structure.

The presence of non-trivial automorphisms already suggests that $\tg$, unlike its ordinary counterpart, is not represented by a complex supermanifold. Nevertheless, it has relatively simple geometry: it is a Deligne--Mumford superstack with an \'etale atlas given by bases of Kuranishi families (Proposition~\ref{super teich is DM}).

There is a natural forgetful map
\[
\pi: \tg \to \mfr
\]
from $\tg$ to the supermoduli stack $\mfr$ of genus $g>1$ super Riemann surfaces. Ordinary Teichm\"uller space has a particularly simple topology, namely that of an open ball. In Section~\ref{sec:teichmuller-theorem}, we prove that $\tg$ admits an \'etale atlas whose underlying topological space is the disjoint union of $2^{2g}$ unit balls (Theorem~\ref{teich theorem for super}).

We prove this statement in three steps. First, we show that the moduli stack $ST_g'$ of continuous families of marked spin curves admits a representable \'etale cover by the disjoint union of $2^{2g}$ unit balls, indexed by the distinct spin structures on a fixed central fiber.
We then use the ordinary Teichm\"uller homeomorphism to define an analytic structure on the corresponding family and show that the resulting classifying map into the moduli stack $ST_g$ of marked spin families is \'etale. It is automatically representable because $ST_g$ is Deligne--Mumford.
Finally, we apply unique \'etale lifting along the representable closed immersion $ST_g \hra \tg$ to obtain an \'etale atlas for $\tg$ whose underlying topological space is the disjoint union of these unit balls.
\medskip

\emph{Future directions.} We would like to compare our construction with the hyperbolic approach to super Teichm\"uller theory, including the decorated theory of Penner and Zeitlin \cite{PZ}. 

\section{Preliminaries}
\label{sec:preliminaries}

\subsection{Conventions}
\label{sec:conventions}

Throughout, $g>1$. In the algebraic portion of this paper (Section~\ref{sec:frames}) we work over $\bb$, in the Zariski or \'etale topology. In the analytic part, we use complex analytic superspaces and the topology of open covers. 

For a superalgebra $A$, we write $A=A_{\ov{0}}\oplus A_{\ov{1}}$. Morphisms preserve the grading unless stated otherwise, and $\Pi$ denotes parity reversal. We write $J$ for the ideal generated by the odd part, and use the same notation for the corresponding ideal sheaf. The bosonic reduction of a superspace $X$ is
\[
X_{\bos}=(|X|,\mc{O}_X/J).
\]
For a functor $\mcx$ on superschemes, $\mcx_{\bos}$ denotes its restriction to ordinary schemes $\on{Sch}\subset\sS$; we use the analogous convention for analytic superspaces and for categories fibered in groupoids.

For a family $\pi:X\to T$ of supercurves, we write $\omega_{X/T}$ for its relative dualizing (Berezinian) bundle. We use the quotient convention for projective superspace: $\mathbb{P}_T(\mce)$ parametrizes locally free quotients of rank $1|0$ of $\mce$.

\subsection{Super Riemann surfaces and spin curves}
\label{sec:srs-spin}

A \emph{supercurve} $X$ is a smooth, proper, and connected superscheme over $\bb$ of dimension $(1|1)$ such that its bosonic reduction $C:= \xbos=(|X|, \ox/\jj)$ is an ordinary smooth, connected curve over $\bb$.
The \emph{genus} $g$ of a supercurve $X$ is the genus $g$ of its underlying bosonic curve $C$. From now on, we assume $g \ge 2$.

    A \emph{family of genus $g$ supercurves over a superscheme $T$} is a smooth, proper morphism of superschemes $\pi: X \to T$ of relative dimension $(1|1)$ such that each geometric fiber $X_t$ is a genus $g$ supercurve over its residue field $k(t)$ and such that the bosonic reduction 
    \[ \pibos: \xbos\to \tbos\]
    is a family of ordinary smooth curves of genus $g$ over $\tbos$.
    An \emph{isomorphism of families over the same base $T$} from $X$ to $X'$ is an isomorphism of superschemes $\phi: X \to X'$ commuting with the projections to $T$.  
\medskip

A genus $g$ \emph{super Riemann surface} is a genus $g$ supercurve $X$ over $\bb$ equipped with a \emph{superconformal structure}, that is, a rank $(0|1)$ locally free subsheaf $\mc{D} \subset \mc{T}_X$ of the tangent sheaf which is maximally non-integrable in the sense that the map
\begin{equation} \label{beta comp}
\mc{D}^{\otimes 2} \lra{[\ , \ ]} \mct_X \to \mct_X / \mc{D},
\end{equation}
(where $\bracket$ is the supercommutator) 
 is an isomorphism of locally free sheaves on $X$. 
This gives a short exact sequence: 
\begin{equation}
    \label{sesofdist}
    0 \lra{{}} \mc{D} \lra{{}} \mct_X \lra{{}} \mc{D}^{\otimes 2} \lra{{}} 0. 
\end{equation}

    A \emph{family of super Riemann surfaces over a superscheme $T$} is a family of supercurves $X\to T$ equipped with a relative superconformal structure $\mc{D} \subset\mct_{X/T}$.
            An \emph{isomorphism of super Riemann surfaces}, called a \emph{superconformal isomorphism}, is a $T$-linear isomorphism of supercurves $\phi:  X \lgr Y$ such that  
            \[
                d\phi(\mc{D}_X) = \phi^*\mc{D}_Y
                \]
                as subsheaves of $\phi^*\mct_{Y/T}$:
              here $d\phi: \mct_{X/T} \lgr \phi^* \mct_{Y/T}$ denotes the induced isomorphism on relative tangent sheaves.

We denote the moduli stack of genus $g$ super Riemann surfaces by $\mfr$. In both the algebraic and analytic categories, it is a smooth, separated Deligne--Mumford superstack of dimension $3g-3|2g-2$. For the algebraic construction, see \cite{ott2026algebraicity}; for the analytic statement, see \cite[Theorem~5.3]{CV}.

Let $T$ be an ordinary scheme or analytic space. A family of \emph{spin curves} is a triple $(\pi:C\to T,L,\alpha)$, where $\pi$ is a smooth proper family of connected genus $g$ curves, $L$ is a line bundle on $C$, and
\[
\alpha:L^{\otimes2}\lgr K_{C/T}.
\]
The same definition applies in the analytic category, with holomorphic line bundles and isomorphisms. An isomorphism from $(C/T,L,\alpha)$ to $(C'/T,L',\alpha')$ consists of an isomorphism $\phi:C\lgr C'$ over $T$ and an isomorphism $\theta:\phi^*L'\lgr L$ satisfying
\[
\alpha\circ\theta^{\otimes2}
=d\phi^*\circ\phi^*\alpha',
\qquad
 d\phi^*:\phi^*K_{C'/T}\lgr K_{C/T}.
\]
There is a well-known equivalence between families of spin curves and families of super Riemann surfaces over purely bosonic bases; see \cite[Lemma~4.2]{ott2026algebraicity} for the algebraic version and \cite[Proposition~4.2]{CV} for the analytic version.
One consequence of this equivalence that we use throughout is the following expression for the structure sheaf of a super Riemann surface as the direct sum $\mc{O}_X=\mc{O}_C\oplus\Pi L$ where $L$ is the spin structure determined by the superconformal structure.

\section{Moduli of framed super Riemann surfaces}
\label{sec:frames}

Let $\pi:X\to T$ be a family of genus $g>1$ super Riemann surfaces, and set
\[
\mcl:=\duaxtsix,
\qquad
\mce:=\pist\mcl.
\]

\begin{lemma}\label{lemma: properties of six-canonical bundle}
The line bundle $\mcl$ is strongly relatively very ample. Furthermore, $\mce$ is locally free of rank $5g-5|6g-6$, and its formation commutes with base change.
\end{lemma}

\begin{proof}
The line bundle $\mcl$ is strongly relatively very ample (see, e.g., \cite{ott2026algebraicity}). 

For every point $t \in T$, let $X_t$ denote the fiber, and let $(C_t,L_t)$ be the spin curve determined by $X_t$. Over a bosonic base, we have the well-known identifications
\[
\mc{O}_{X_t}=\mc{O}_{C_t}\oplus\Pi L_t,
\qquad
\mcl\vert_{X_t}=L_t^6\oplus\Pi L_t^7,
\qquad
L_t^2\cong\omega_{C_t}.
\]
Since $\deg L_t=g-1$, both summands of $\mcl\vert_{X_t}$ have vanishing $H^1$. The ordinary Riemann--Roch formula gives
\[
h^0(C_t,L_t^6)=5g-5,
\qquad
h^0(C_t,L_t^7)=6g-6.
\]
Cohomology and base change now show that $\mce$ is locally free of this rank and commutes with base change.

\end{proof}

Henceforth, we set
\[
V=\bb^{5g-5|6g-6},
\qquad
V_T=V\otimes_{\bb}\ot,
\qquad
\mce_{X/T}=\pist\mcl_{X/T}.
\]
\begin{definition}
We define a \emph{frame} on $X/T$ to be a global section of the following quotient sheaf:
\[
\underline{\on{Isom}}_{\ot}(V_T,\mce_{X/T})
\big/\mc{O}_{T,0}^{*},
\]
where invertible even functions on the base act by scalar multiplication on $V_T$. We write $\Fr(X/T)$ for the set of all frames on $X/T$.
\end{definition}

Every frame $\rho\in\Fr(X/T)$ has local representatives $(T_i,\rho_i)$, where $\{T_i\}$ is an open cover of $T$ and
\[
\rho_i:V_{T_i}\lgr\mce_{X/T}\vert_{T_i}
\]
is an isomorphism. On $T_i\cap T_j$, these representatives satisfy $\rho_i=\lambda_{ij}\rho_j$ for some invertible even function $\lambda_{ij}$.
Two collections $(T_i,\rho_i)$ and $(T_i',\rho_i')$ represent the same frame if their restrictions to a common refinement differ by invertible even scalar functions.

For any morphism $f:S\to T$ of superschemes, we have a pullback map
\[
\fpull:\Fr(X/T)\longrightarrow\Fr(X_S/S)
\]
defined as follows. Given local representatives $(T_i,\rho_i)$ of $\rho$, set $S_i:=f^{-1}(T_i)$ and $f_i:=f|_{S_i}:S_i\to T_i$. The pulled-back frame is represented by $(S_i,f_i^*\rho_i)$, where
\[
f_i^*\rho_i:V_{S_i}\lgr f_i^*\mce_{X_{T_i}/T_i}
\cong\mce_{X_{S_i}/S_i}.
\]
This is well-defined because pullback commutes with scalar multiplication.
Using these pullback maps, we define the \emph{framing sheaf} on the big Zariski site of $T$:
\[
\underline{\Fr}_{X/T}:(\sS/T)^{\on{op}}\longrightarrow\on{Set},
\qquad
(S\to T)\longmapsto\Fr(X_S/S).
\]

The supergroup scheme $G=\on{PGL}(V)$ acts naturally on the framing sheaf on the right:
\begin{equation}\label{action map on frame}
\mu_X:\underline{\Fr}_{X/T}\times_T G_T
\longrightarrow\underline{\Fr}_{X/T}.
\end{equation}
For any morphism $S\to T$, we write this action as
\begin{equation}\label{s point of frame action}
\rho\cdot g:=\rho\circ g,
\qquad
\rho\in\Fr(X_S/S),\quad g\in G_T(S).
\end{equation}
Here composition is understood using local representatives: choose $(S_i,\rho_i)$ representing $\rho$ and lifts $\tilde g_i\in\on{GL}(V)(S_i)$ of $g|_{S_i}$. Then $\rho\cdot g$ is represented by $(S_i,\rho_i\circ\tilde g_i)$. Changing the representatives or the lifts multiplies these compositions by invertible even scalars, so the resulting frame is well-defined.

\begin{lemma}\label{lem:frame-torsor}
For any family of genus $g>1$ super Riemann surfaces
$\pi:X\to T$, its framing sheaf $\underline{\Fr}_{X/T}$
is a $G_T$-torsor on the big Zariski site of $T$ and is
represented by a smooth superscheme of finite type over $T$.
\end{lemma}

\begin{proof}
We need to check that $\ufrxt$ satisfies the two conditions
of a $G_T$-torsor on the big Zariski site.

\medskip
\emph{Locally non-empty:}
We must show that for every morphism $f:S\to T$ there exists
a Zariski cover $\bigcup_i S_i=S$ such that
$\ufrxt(S_i)=\Fr(X_{S_i}/S_i)\neq\emptyset$. This is equivalent to giving an open cover of $S$ on which 
the vector bundle $\fpull\mce=\mc{E}_{X_S/S}$ is trivial, which exists because the bundle is locally free.
\medskip

\emph{$G_T$ acts freely and transitively:}
Suppose $S\to T$ is a morphism for which
$\ufrxt(S)\neq\emptyset$.
Let $\rho\in\ufrxt(S)$, and suppose that $\rho\cdot g=\rho$
for some $g\in G_T(S)$.
Choose local representatives $(S_i,\rho_i)$ of $\rho$ and lifts
$\tilde{g}_i\in\on{GL}(V)(S_i)$ of $g\vert_{S_i}$. 
Since $\rho\cdot g=\rho$, there are even invertible functions $\lambda_i$ on $S_i$ such that
\[
\rho_i\circ\tilde{g}_i=\lambda_i\rho_i.
\]
Composing with $\rho_i^{-1}$ gives
$\tilde{g}_i=\lambda_i\id_{V_{S_i}}$.
Hence $g\vert_{S_i}=\id$ for every $i$, and thus $g=\id$ on $S$.

For transitivity, let $\rho,\rho'\in\ufrxt(S)$, and choose representatives
$(S_i,\rho_i)$ and $(S_i,\rho_i')$.
Set
$\tilde{g}_i:=\rho_i'^{-1}\circ\rho_i
\in\on{Aut}(V_{S_i})$, and let $g_i\in G_T(S_i)$
be its projective class. 
Changing the representatives multiplies $\tilde{g}_i$
by an even invertible scalar, so it does not change $g_i$.

On $S_i\cap S_j$, write
$\rho_i=\lambda_{ij}\rho_j$ and
$\rho_i'=\lambda_{ij}'\rho_j'$, where
$\lambda_{ij}$ and $\lambda_{ij}'$ are even invertible
functions. Then
\[
\tilde{g}_i
=
(\lambda_{ij}')^{-1}\lambda_{ij}\tilde{g}_j.
\]
Thus $g_i$ and $g_j$ agree on $S_i\cap S_j$ and glue to
$g\in G_T(S)$.
The identities $\rho_i'\circ\tilde{g}_i=\rho_i$ give
$\rho'\cdot g=\rho$.

\medskip
For representability, take $S=T$ and choose a Zariski cover
$\bigcup_i T_i=T$ with frames $\rho_i\in\ufrxt(T_i)$.
Freeness and transitivity give isomorphisms of sheaves
\[
\theta_i:G\times T_i\xrightarrow{\sim}\ufrxt\vert_{T_i}.
\]

Write $T_{ij}:=T_i\cap T_j$.
On $T_{ij}$, let $g_{ji}\in G(T_{ij})$ be the unique
section satisfying $\rho_j\cdot g_{ji}=\rho_i$.
The transition map $\theta_j^{-1}\circ\theta_i$ is
left multiplication by $g_{ji}$, hence is an isomorphism
of superschemes
$\psi_{ji}:G\times T_{ij}\to G\times T_{ij}$.
On $T_i\cap T_j\cap T_k$, we have
\[
\rho_k\cdot(g_{kj}g_{ji})
=
\rho_j\cdot g_{ji}
=
\rho_i
=
\rho_k\cdot g_{ki}.
\]
Freeness gives $g_{kj}g_{ji}=g_{ki}$, so
$\psi_{kj}\circ\psi_{ji}=\psi_{ki}$.
Also $\psi_{ii}=\id$ and $\psi_{ij}=\psi_{ji}^{-1}$.

The superschemes $G\times T_i$ therefore glue along the
maps $\psi_{ji}$ to a superscheme $Y\to T$.
The identities $\theta_j\circ\psi_{ji}=\theta_i$ show that
the $\theta_i$ glue to an isomorphism
$\on{Hom}_T(-,Y)\cong\ufrxt$.
Thus $Y$ represents $\ufrxt$.

Finally, each restriction $Y\times_T T_i\to T_i$ is
isomorphic to $G\times T_i\to T_i$, which is smooth and
of finite type because $G$ is smooth and of finite type.
Both properties are Zariski local on the target, so
$Y\to T$ is smooth and of finite type.
\end{proof}

Our next goal is to define framed super Riemann surfaces and
their moduli. Before doing so, we explain how framing sheaves
transform under superconformal isomorphisms.
\medskip

Let $\pi:X\to T$ and $\pi':X'\to T$ be two families of super
Riemann surfaces, and let $\phi:X\lgr X'$ be a superconformal
isomorphism over $T$. We define an induced isomorphism
\[
\phipull:\underline{\Fr}_{X'/T}\lgr\ufrxt.
\]
Write $\gamma_\phi:\phipull\mcl_{X'/T}\lgr\mcl_{X/T}$
for the canonical isomorphism, and define
\[
c_\phi:
\mce_{X'/T}
=\pi'_*\mcl_{X'/T}
\lgr\pi_*\phipull\mcl_{X'/T}
\xrightarrow{\pi_*\gamma_\phi}
\pi_*\mcl_{X/T}
=\mce_{X/T}.
\]
The first arrow is induced by pulling back sections along
$\phi$, using $\pi'\circ\phi=\pi$.

For any morphism $f:S\to T$ and any frame
$\rho'\in\underline{\Fr}_{X'/T}(S)$, choose representatives
$(S_i,\rho_i')$. Define $\phipull\rho'$ to be the frame in
$\ufrxt(S)$ represented by
$(S_i,(\fpull c_\phi)\vert_{S_i}\circ\rho_i')$.
Multiplying $\rho_i'$ by an even invertible function
multiplies its image by the same function. Thus this
definition is independent of the representatives. The same construction for
$\phi^{-1}$ gives its inverse.

\begin{definition}
A \emph{family of framed super Riemann surfaces over $T$}
is a triple $(\pi:X\to T,D,\rho)$, where $(\pi,D)$ is a
family of super Riemann surfaces over $T$ and
$\rho\in\ufrxt(T)$.
Thus a given family admits a framing precisely when
$\ufrxt(T)\neq\emptyset$.
An isomorphism from $(\pi,D,\rho)$ to $(\pi',D',\rho')$
is a superconformal isomorphism $\phi:X\lgr X'$ over $T$
satisfying $\phipull\rho'=\rho$.
\end{definition}
We let $P$ denote the category fibered in groupoids of
framed super Riemann surfaces, with pullback given by
pullback of the family and its frame.

\begin{lemma}
    $P$ is an \'etale superstack.
\end{lemma}
\begin{proof}
To verify \'etale descent, let $\{T_i\to T\}$ be an
\'etale cover, and let $(X_i/T_i,D_i,\rho_i)$ be framed super Riemann surfaces equipped with isomorphisms on overlaps satisfying the cocycle condition. Since super Riemann surfaces satisfy \'etale descent, the underlying families glue to a family $(X/T,D)$ with isomorphisms
$\theta_i:X_{T_i}\lgr X_i$. The isomorphisms on the overlaps $T_i\times_T T_j$ preserve the frames, so the pullbacks
$\theta_i^*\rho_i\in\ufrxt(T_i)$ agree on all overlaps. Since $\ufrxt$ is represented by a superscheme, it is an \'etale sheaf. Hence these sections glue to a unique global frame
$\rho\in\ufrxt(T)$.

Isomorphisms also descend since superconformal isomorphisms
satisfy \'etale descent, and the condition
$\phipull\rho'=\rho$ can be checked after an \'etale
cover. 
\end{proof}

Let $\mfr$ be the moduli superstack of genus $g$ super
Riemann surfaces. It is smooth, Deligne--Mumford, and of
finite type over $\bb$. Forgetting the frame gives a morphism
\[
F:P\longrightarrow\mfr.
\]

\begin{prop}\label{prop: P is DM}
The morphism $F$ is smooth, of finite type, and representable
by superschemes. In particular, $P$ is a smooth
Deligne--Mumford superstack of finite type over $\bb$.
\end{prop}

\begin{proof}
We must check that for every morphism $T\to\mfr$ from a
superscheme $T$, the base change $P_T:=P\times_{\mfr}T$
is represented by a superscheme smooth and of finite type
over $T$. We will show that $P_T$ is naturally equivalent
to the framing sheaf $\ufrxt$ of the family $X/T$
classified by $T\to\mfr$, viewed as a stack with only
identity arrows.

For any $S\to T$, $P_T(S)$ is the groupoid whose objects
are triples $(Y/S,\rho,\alpha)$ consisting of a framed
super Riemann surface $(Y/S,\rho)$ and a superconformal
isomorphism $\alpha:X_S\lgr Y$ over $S$.
Define
\[
\Phi_S:P_T(S)\longrightarrow\ufrxt(S),
\qquad
(Y/S,\rho,\alpha)\longmapsto\alpha^*\rho.
\]

\medskip
\noindent\emph{Essentially surjective.}
For any $\rho'\in\ufrxt(S)$, the object
$(X_S/S,\rho',\id_{X_S})$ maps to $\rho'$.

\medskip
\noindent\emph{Fully faithful.}
A morphism from $(Y/S,\rho,\alpha)$ to
$(Y'/S,\rho',\alpha')$ is an isomorphism
$\phi:(Y/S,\rho)\lgr(Y'/S,\rho')$ of framed super
Riemann surfaces satisfying
\begin{equation}\label{expression for phi}
\phi\circ\alpha=\alpha'
\iff
\phi=\alpha'\circ\alpha\inv.
\end{equation}
Thus there is at most one such morphism.
The superconformal isomorphism
$\phi=\alpha'\circ\alpha\inv$ preserves the frames
precisely when $\phipull\rho'=\rho$.
Since $\alpha^*$ is a bijection on frames, we have
\[
\phipull\rho'=\rho
\iff
\alpha^*(\phipull\rho')=\alpha^*\rho
\iff
\alpha'^*\rho'=\alpha^*\rho.
\]
Consequently, the morphism set between these two objects
is a singleton exactly when their images under $\Phi_S$
are equal, and is empty otherwise. This proves full
faithfulness.

By Lemma~\ref{lem:frame-torsor}, $\ufrxt$ is represented
by a superscheme smooth and of finite type over $T$.
This proves the assertions about $F$.

Finally, an \'etale atlas $U\to\mfr$ pulls back to an
\'etale atlas $P\times_{\mfr}U\to P$ by a superscheme.
Hence $P$ is Deligne--Mumford. Its smoothness and finite
type over $\bb$ follow by composing $F$ with the smooth
finite-type structure morphism of $\mfr$.
\end{proof}

As a corollary, we obtain the following global quotient description of supermoduli space: 
\begin{corollary}\label{global quotient}
There is a canonical equivalence
\[
\mfr\cong[P/G].
\]
\end{corollary}

Next we prove that $P$ is an algebraic superspace by showing that its objects admit no nontrivial automorphisms. To do so, we use the closed immersions induced by frames and introduce notation used throughout the rest of the paper.

Let $(\pi: X \to T, D, \rho)$ be an object of $P(T)$, and choose local representatives $(T_i,\rho_i)$ of $\rho$. Each $\rho_i$ determines a surjection
\begin{equation}\label{q rho}
q_{\rho_i}: V_{X_{T_i}} \twoheadrightarrow \mcl|_{X_{T_i}}
\end{equation}
and hence a map $i_{\rho_i}:X_{T_i}\to\pp(V)$. On overlaps, the representatives differ by invertible even scalars, so the maps $i_{\rho_i}$ agree and define a map $\irho:X\to\pp(V)$. Write
\begin{equation}\label{j rho the closed immersion}
j_{\rho}:=(\irho,\pi):X\hra\pp(V)\times T
\end{equation}
for the resulting closed immersion over $T$, and set
\[
\mcz_{\rho}:=j_{\rho}(X)\subset\pp(V)\times T.
\]
Then $j_{\rho}$ factors as
\begin{equation}\label{psi rho}
X\overset{\psi_{\rho}}{\cong}\mcz_{\rho}
\subset\pp(V)\times T,
\end{equation}
where the second map is the inclusion.

\begin{lemma}\label{P is alg}
The superstack $P$ is a smooth algebraic superspace of finite type over $\bb$.
\end{lemma}

\begin{proof}
By Proposition~\ref{prop: P is DM}, it suffices to show that every automorphism of a framed family is the identity.

Let $\phi:X\cong X$ be a superconformal automorphism satisfying $\phipull\rho=\rho$. For each local representative $\rho_i$, we have
\[
q_{\phi^*\rho_i}
=\gamma_{\phi}\circ\phipull q_{\rho_i}.
\]
Consequently, $j_{\phipull\rho}=j_{\rho}\circ\phi$. Since $\phipull\rho=\rho$, this gives $j_{\rho}=j_{\rho}\circ\phi$. The closed immersion $j_{\rho}$ is a monomorphism, so $\phi=\id_X$.
\end{proof}

Next we show that $P$ is a smooth superscheme of finite type over $\bb$. Its analytification is therefore a complex supermanifold, which we will use in the section on super Teichm\"uller space.
\medskip

We recall the Hilbert superscheme from \cite{bruzzo2023notes}. Let $H$ be the contravariant functor from $\sS$ to $\on{Set}$ sending $T$ to the set of closed subsuperschemes $Z\subset\pp(V)\times T$ such that $Z\to T$ is a family of smooth genus $g$ supercurves and $\mc{O}_Z(1)$ has super Hilbert polynomial
\[
6n(g-1)+1-g\;|\;6n(g-1).
\]
The Hilbert superscheme exists by \cite[Theorem~4.3]{bruzzo2023notes}; $H$ is its open locus parametrizing these smooth supercurves. Write $\mcz\to H$ for the universal family. We do not impose a six-canonical condition on $H$.

For a morphism $f:T\to H$, we denote the pullback of $\mcz$ by $\mcz_T$, or by $\mcz_f$ when we want to emphasize the map.
\medskip

Next we include the data of a superconformal structure. Let $Q$ be the contravariant functor on $\sS/H$ sending $f:T\to H$ to the set of isomorphism classes of locally free rank $1|0$ quotients
\[
q:\mct_{\mcz_T/T}\twoheadrightarrow M.
\]
Two quotients $q$ and $q'$ are isomorphic if there is an isomorphism $\phi:M\cong M'$ satisfying $\phi\circ q=q'$. The functor $Q$ is represented by a superscheme over $H$. 

For such a quotient, set $D:=\ker(q)$, a locally free rank $0|1$ subbundle of $\mct_{\mcz_T/T}$. Define $Q'\subset Q$ by requiring that the map induced by the bracket,
\[
D\otimes D\xra{\bracket}\mct_{\mcz_T/T}/D,
\]
is an isomorphism. In the universal curve over $Q$, the locus where this map is not an isomorphism is closed. Since the universal curve is proper over $Q$, its image in $Q$ is closed. The complement of this image represents $Q'$, so $Q'$ is an open subsuperscheme of $Q$ \cite{ADT}. 

Since $M\cong\mct_{\mcz_T/T}/D$, we may equivalently describe a point of $Q'$ by the pair $(\mcz_T,D)$.
\medskip

There is a natural morphism
\[
I:P\to Q'
\]
sending $(\pi:X\to T,D,\rho)$ to $(\mcz_{\rho},d\psi_{\rho}(D))$. Here $d\psi_{\rho}(D)$ denotes the image of $\psi_{\rho*}D$ under
\[
d\psi_{\rho}:
\psi_{\rho*}\mct_{X/T}\lgr\mct_{\mcz_{\rho}/T}.
\]

We use the bosonic-restriction convention of Section~\ref{sec:conventions}.

\begin{lemma}\label{lem:bosonic-monomorphism}
The morphism
\[
I_{\bos}:P_{\bos}\longrightarrow Q'_{\bos}
\]
is a monomorphism of algebraic spaces.
\end{lemma}

\begin{proof}
Let $T$ be an ordinary scheme, and suppose
\[
(\mcz_{\rho},d\psi_{\rho}(D))
=(\mcz_{\rho'},d\psi_{\rho'}(D'))
\]
in $Q'(T)$. The equality $\mcz_{\rho}=\mcz_{\rho'}$ defines an isomorphism
\[
\phi:=\psi_{\rho'}\inv\circ\psi_{\rho}:X\lgr X'.
\]
The equality of the transported distributions makes $\phi$ superconformal, and the factorizations of $j_{\rho}$ and $j_{\rho'}$ give
\[
j_{\rho}=j_{\rho'}\circ\phi.
\]

We must show that $\phipull\rho'=\rho$. This can be checked Zariski locally on $T$, so choose local linear representatives of both frames, and denote their evaluation maps by $q_{\rho}$ and $q_{\rho'}$.

To compare the frames, we express the equality of embeddings in terms of these evaluation maps. By the quotient description of $\pp(V)$, the equality $j_{\rho}=j_{\rho'}\circ\phi$ gives a unique isomorphism $\lambda:\phipull\mcl'\lgr\mcl$ satisfying
\[
\lambda\circ\phipull q_{\rho'}=q_{\rho}.
\]
The evaluation map corresponding to the pulled-back frame is $\gamma_{\phi}\circ\phipull q_{\rho'}$, where $\gamma_{\phi}:\phipull\mcl'\lgr\mcl$ is the canonical isomorphism. Thus $\phipull\rho'=\rho$ if and only if
\[
\lambda=\pipull(a)\gamma_{\phi}
\]
for some $a\in\Gamma_T^*:=H^0(T,\mc{O}_T)^{\times}$.

The automorphism $\gamma_{\phi}\circ\lambda\inv$ of $\mcl$ is multiplication by an element
\[
s\in\Gamma_X^*:=H^0(X,\mc{O}_X)_{\ov{0}}^{\times}.
\]
The framing condition is therefore equivalent to $s$ lying in the image of $\pipull:\Gamma_T^*\to\Gamma_X^*$.

Let $C=X_{\bos}$, and let $J$ be the odd ideal of $\ox$. Since $T$ is ordinary and $X/T$ has relative dimension $1|1$, we have $J^2=0$ and the canonical decomposition
\[
\ox=\oc\oplus J,
\qquad
(\ox)_{\ov{0}}=\oc.
\]
Moreover, $C\to T$ is a smooth proper family of geometrically connected curves, so $(\pi|_C)_*\oc=\mc{O}_T$. It follows that pullback gives an isomorphism
\[
\pipull:
H^0(T,\mc{O}_T)^{\times}
\lgr H^0(X,\mc{O}_X)_{\ov{0}}^{\times}.
\]
Thus $s=\pipull(a)$ for some $a\in\Gamma_T^*$, and $\lambda=\pipull(a^{-1})\gamma_{\phi}$. Hence $\phi$ preserves the frames locally, and therefore globally.

Consequently, equality of the images in $Q'(T)$ implies that the two objects are isomorphic in $P(T)$. Since this holds for every ordinary scheme $T$, $I_{\bos}$ is a monomorphism.
\end{proof}

\begin{prop}\label{prop:P-superscheme}
The morphism $I:P\to Q'$ is locally quasi-finite and separated. Moreover, $P$ is a smooth superscheme of finite type over $\bb$.
\end{prop}

\begin{proof}
Since $P$ is of finite type over $\bb$, the morphism $I$ is locally of finite type. Its bosonic restriction $I_{\bos}$ is a monomorphism, hence is separated and locally quasi-finite. These two properties are detected on the bosonic reduction, so $I$ is separated and locally quasi-finite.

By \cite[Proposition~A.1]{POP}, a separated locally quasi-finite morphism from an algebraic superspace to a superscheme is representable by superschemes. Applying this to $I:P\to Q'$ shows that $P$ is a superscheme. Its smoothness and finite-type property over $\bb$ were established in Lemma~\ref{P is alg}.
\end{proof}

\section{Super Teichm\"uller space} \label{section: super Teich space}

We now work with complex-analytic superspaces. We will continue to write $P$ and $G:= \on{PGL}(5g-5|6g-6,\bb)$ for their analytifications. Since both are smooth and of finite type over $\bb$, both are complex supermanifolds after analytification.  

\subsection{Kuranishi families from moduli of framed super Riemann surfaces}

In this section, we construct Kuranishi families for super Riemann surfaces using $P$ as an ambient space. 
\medskip

We start by defining the $G$-action on $P$.  Let 
\begin{equation} \label{universal framed family} (\mcz \to P, \rho)\end{equation}
be the universal framed family. For every map $t:T\to P$, write $(\mcz_T \to T, \ t^*\rho)$ for the pullback of \eqref{universal framed family}. For any $g:T\to G$, set $t_g:T\to P$ to be the classifying map of the framed family
\[
(\mcz_T \to T, \ (t^*\rho) \circ g),
\]
and define
\[
a_T:G(T)\times P(T)\to P(T),
\qquad
a_T(g,t):=t_g.
\]
These maps commute with pullback and therefore define a morphism $a:G\times P\to P$.

Next we identify the stabilizer of a point with the superconformal automorphism group of the corresponding super Riemann surface.

Let $p\in P$ correspond to $(X,\rho)$. An element $g\in G$ fixes $p$ if and only if there is a superconformal automorphism $\gamma:X\lgr X$ satisfying $\gamma^*(\rho\circ g)=\rho$. By definition of the pullback of a frame, this means
\[
c_\gamma\circ\rho\circ g=\rho,
\qquad
g=\rho^{-1}\circ c_\gamma^{-1}\circ\rho,
\]
where equality is up to scalar multiplication. 
Writing $\Aut_{\mfr}(X)$ for the superconformal automorphism group of $X$, the above formulas give a bijection: 
\[
\Aut_{\mfr}(X)\to\on{Stab}_G(p),
\qquad
\gamma\mapsto\rho^{-1}\circ c_\gamma^{-1}\circ\rho.  
\]
In particular, $\on{Stab}_G(p)$ is finite and bosonic because $\on{Aut}_{\mfr}(X)$ is
\cite[Proposition~4.3]{CV}: indeed, the tangent space to the group of superconformal automorphisms of $X$ is the automorphism group of first-order deformations of $X$, the elements of which are usually referred to as infinitesimal superconformal automorphisms. The automorphisms of first-order deformations are in bijection with $H^0(X,\mc{A}_X)$, where $\mc{A}_X$ is the sheaf of superconformal vector fields. The identification $\mc{A}_X\cong\mc{D}^2$, as sheaves of complex vector spaces, together with the decomposition $\mc{D}^2= T_C \oplus \Pi (T_C \otimes L)= T_C \oplus \Pi L \inv$, gives $H^0(X, \mc{A}_X)=0$. 
\medskip

We now define deformations and Kuranishi families.
A \emph{deformation} of $X$ over a pointed analytic superspace $(S,s_0)$ is a family $\eta:Y\to S$ together with a superconformal isomorphism $\psi:X\lgr Y_{s_0}$. A morphism of deformations of $X$ from $(\eta:Y\to(S,s_0),\psi)$ to $(\eta':Y'\to(S',s_0'),\psi')$ is a cartesian pair $(f,F)$ 
\[
\begin{tikzcd}
Y \arrow[r,"F"] \arrow[d,"\eta"'] &
Y' \arrow[d,"\eta'"] \\
(S,s_0) \arrow[r,"f"] &
(S',s_0')
\end{tikzcd}
\]
such that $f(s_0)=s_0'$, the induced isomorphism
\[ \tilde{F}: Y \lgr Y' \times_{S',f} S\]
preserves the superconformal structure and $\tilde{F}_{s_0}: Y_{s_0} \lgr Y_{s_0'}'$ satisfies $\tilde{F}_{s_0} \circ \psi=\psi'$.

A \emph{Kuranishi family} for $X$ is a deformation
\[
\pi:\mcx\to(B,b_0),
\qquad
\phi:X\lgr\mcx_{b_0},
\]
where $B$ is a complex supermanifold of dimension $3g-3|2g-2$, with the following property: If $(\eta: Y \to (S,s_0), \psi)$ is a deformation of $X$ then there exists a morphism of $X$-deformations $(f,F)$ from $(\eta, \psi)$ to $(\pi, \phi)$ which is essentially unique in the sense that if $(g,G)$ is another such morphism then there exists an open neighborhood $W \ni s_0$ of $S$ such that $(i_W, I_W) \circ (f,F)=(i_W, I_W) \circ (g,G)$ where $(i_W: W \subset S, I_W)$ is the induced cartesian pair 
\[ \begin{tikzcd}
    \eta \inv(W) \arrow[r, "I_W"] \arrow[d, "\eta \vert_W"] & Y \arrow[d, " \eta"] \\
    W \arrow[r, "i_W"] & S.
\end{tikzcd}\]

Fix $p\in P$, and let $(X,\rho)$ be the corresponding framed super Riemann surface. We construct a Kuranishi family for $X$ by choosing a slice through $p$ transverse to its $G$-orbit.

Let $\frak{o}(p)$ be the orbit of $p$. The orbit is identified with the quotient complex supermanifold
$G/\on{Stab}_G(p)$. The canonical map
$
G/\on{Stab}_G(p) \to P$
is an injective immersion, and so the induced map on tangent spaces $T_p\frak{o}(p)\subset T_pP$ is injective. Choose a complementary subspace $N_p$ to the orbit tangent space and a local subsupermanifold $U\subset P$ through $p$ with $T_pU=N_p$.

Let $(\pi:\mcx\to U,\rho_U)$ be the restriction of \eqref{universal framed family} to $U$. The universal frame $\rho$ determines the canonical isomorphism
\[
\phi:=\phi_\rho:X\lgr \mcx \vert_p.
\]
Thus the triple 
\begin{equation}\label{kuranishi}
(\pi:\mcx\to(U,p),\ \rho_U, \ \phi).
\end{equation}
is a deformation of $X$.

\begin{lemma}\label{slice is kuranishi}
The family \eqref{kuranishi} is a Kuranishi family for $X$.
\end{lemma}

\begin{proof}

We first note that the base $U$ has the correct dimension: $\dim(U)=3g-3|2g-2$. Indeed, this follows from the equivalence $\mfr\cong[P/G]$, together with the isomorphism $T_p \frak{o}(p) \cong T_e G$ which follows from $T_e \on{Stab}_G(p)=0$.  

We also note that the differential of the restricted action $
a|_{G\times U}:G\times U\to P
$
is invertible at $(e,p)$. Indeed,
\[
da_{e,p}:T_eG\oplus T_pU\to T_pP,
\qquad
(\xi,v)\mapsto d(a(-,p))_e(\xi)+v,
\]
and its two summands map isomorphically onto the complementary subspaces
$T_p\frak{o}(p)$ and $T_pU$.
The inverse function theorem gives open neighborhoods $W\subset G$ of $e$ and $V\subset P$ of $p$ such that
\[
a|_{W\times U}:W\times U\lgr V.
\]

Let $(\eta:Y\to(S,s_0),\psi)$ be a deformation of $X$. On a neighborhood $S'$ of $s_0$, choose a frame $\rho_Y$ satisfying
\[
\psi^*(\rho_Y|_{s_0})=\rho.
\]
The corresponding map $\alpha:S'\to P$ sends $s_0$ to $p$.


After shrinking $S'$, assume that $\alpha(S')\subset V$, and write
\[
(g,f):=(a|_{W\times U})^{-1}\circ\alpha:
S'\to W\times U.
\]
Then $g(s_0)=e$, $f(s_0)=p$, and $\alpha=a(g,f)$. Thus the two framed families classified by $\alpha$ and $a(g,f)$ admit a unique framed isomorphism
\[
\theta:(Y|_{S'},\rho_Y)
\lgr(f^*\mcx,f^*\rho_U\circ g).
\]
At $s_0$, both $\theta_{s_0}\circ\psi$ and $\phi$ are framed isomorphisms from $(X,\rho)$ to $(\mcx_p,\rho_U|_p)$. Since framed objects have no non-trivial automorphisms,
\[
\theta_{s_0}\circ\psi=\phi.
\]
Set $F:=\pr_{\mcx}\circ\theta:Y|_{S'}\to\mcx$. Then $(f,F)$ is the required cartesian pair.

For uniqueness, first replace $\rho_Y$ by another extending frame $\rho_Y'$ with the same restriction to the central fiber. The torsor property gives a unique map $h:S'\to G$ satisfying
\[
\rho_Y'=\rho_Y\circ h,
\qquad
h(s_0)=e.
\]
Its classifying map is
\[
\alpha'=a(h,\alpha)=a(gh,f).
\]
After shrinking $S'$ so that $gh$ takes values in $W$, the inverse of $a|_{W\times U}$ therefore gives the pair $(gh,f)$. In particular, the base map $f$ is unchanged.

The isomorphism constructed from $\rho_Y'$ satisfies
\[
\theta'^*(f^*\rho_U\circ gh)=\rho_Y\circ h.
\]
Equivariance of pullback gives
\[
\theta'^*(f^*\rho_U\circ g)\circ h=\rho_Y\circ h,
\]
and hence $\theta'^*(f^*\rho_U\circ g)=\rho_Y$. Thus $\theta'$ and $\theta$ are isomorphisms between the same framed families. They are equal, so $F$ is unchanged.

Now let $(f',F')$ be any other cartesian pair preserving the central identification. The frame $F'^*\rho_U$ on $Y$ satisfies
\[
\psi^*((F'^*\rho_U)|_{s_0})
=(F'_{s_0}\circ\psi)^*(\rho_U|_p)
=\phi^*(\rho_U|_p)
=\rho.
\]
Its classifying map is $f':S'\to U\subset P$. After shrinking $S'$, this map takes values in $V$, and its product decomposition is $(e,f')$.

Since $(f',F')$ is cartesian, it induces a framed isomorphism
\[
(Y|_{S'},F'^*\rho_U)
\lgr(f'^*\mcx,f'^*\rho_U).
\]
This is the unique framed isomorphism used in the construction, so that construction returns $(f',F')$. The comparison of extending frames above now gives $(f',F')=(f,F)$ after restricting the source.
\end{proof}

The frame $\rho$ determines the point $p\in P$ and the central identification $\phi_\rho:X\lgr\mcx_p$. The construction also uses a choice of transverse slice through $p$.

We now compare this construction with the one obtained from another frame $\rho'=\rho\circ g$ on $X$. There are two cases.
\begin{enumerate}
\item
If $g\in\on{Stab}_G(p)$, the framed super Riemann surfaces $(X,\rho\circ g)$ and $(X,\rho)$ determine the same point $p\in P$. We use the same family over $U$, with the new central identification:
\begin{equation}\label{second kuranishi}
(\pi:\mcx\to(U,p), \ \rho_U, \ 
\phi_{\rho\circ g}:X\lgr\mcx_p).
\end{equation}

Let $\gamma\in\Aut_{\mfr}(X)$ correspond to $g$. Since
$\gamma:(X,\rho)\lgr(X,\rho\circ g)$ is a framed isomorphism, both $\phi_{\rho\circ g}\circ\gamma$ and $\phi_\rho$ are framed isomorphisms to the same universal fiber. Consequently,
\[
\phi_{\rho\circ g}\circ\gamma=\phi_\rho,
\qquad
\phi_{\rho\circ g}=\phi_\rho\circ\gamma^{-1}.
\]

\item
If $g\notin\on{Stab}_G(p)$, the frame $\rho\circ g$ determines a different point $p'\in P$. Choose a slice $(U',p')\subset P$ through $p'$ transverse to $\frak{o}(p')$, and consider
\begin{equation}\label{third kuranishi}
(\pi':\mcx'\to(U',p'),\rho_{U'},
\phi_{\rho\circ g}:X\lgr\mcx'_{p'}).
\end{equation}
\end{enumerate}

Applying Lemma~\ref{slice is kuranishi} to the frame $\rho\circ g$ shows that both \eqref{second kuranishi} and \eqref{third kuranishi} are Kuranishi families for $X$. The universal property of Kuranishi families then gives the following statement: 

\begin{lemma}\label{universality of Kuranshi}
Let $g\in\on{Stab}_G(p)$, and let $\gamma$ correspond to $g$ under
\[
\on{Stab}_G(p)\cong\Aut_{\mfr}(X).
\]
Possibly shrinking $U$ around $p$, there is an isomorphism of deformations $(\fgam,\Fgam)$ from \eqref{kuranishi} to \eqref{second kuranishi}.

\end{lemma}

We can prove a similar statement for group elements $g\notin\on{Stab}_G(p)$: specifically, there is an isomorphism of deformations $(f,F)$ from \eqref{kuranishi} to \eqref{third kuranishi}, after restricting $U$ and $U'$ to neighborhoods of $p$ and $p'$, satisfying
\[
f(p)=p',
\qquad
\pi'\circ F=f\circ\pi,
\qquad
F_p\circ\phi_\rho=\phi_{\rho\circ g}.
\]

There are two further properties of Kuranishi families that we will need: 

\begin{prop}\label{properties of automorphisms}
After shrinking $U$ around $p$, the following hold.
\begin{enumerate}
\item
The family \eqref{kuranishi} is Kuranishi at every point $u\in U$, with central identification $\id_{\mcx_u}$.

\item
For every $u,u'\in U$ and every superconformal isomorphism
$\theta:\mcx_u\lgr\mcx_{u'}$, there is a unique
$\gamma\in\Aut_{\mfr}(X)$ such that the lift from
Lemma~\ref{universality of Kuranshi} satisfies
\[
\fgam(u)=u',
\qquad
\Fgam|_{\mcx_u}=\theta.
\]
\end{enumerate}
\end{prop}

\subsection{The definition of super Teichm\"uller space} \label{sec: definition of STM}

Fix an oriented topological surface $\Sigma$ of genus $g>1$. A \emph{Teichm\"uller structure}, or marking, on a super Riemann surface $X$ is an isotopy class $[f]$ of orientation-preserving homeomorphisms
\[
f:X\to\Sigma.
\]
An isomorphism of marked super Riemann surfaces
$(X,[f])\lgr(X',[f'])$ is a superconformal isomorphism
$\phi:X\lgr X'$ satisfying
\[
[f'\circ\phi]=[f].
\]
Note that the marking is purely topological data, so a marking on a super Riemann surface is just a marking on its underlying ordinary family of curves. 
\medskip 

For a family $\eta:X\to S$, a \emph{pre-marking} is a collection
$
\{[f_s: X_s\to\Sigma]\}_{s\in S}.$ 
A pre-marking is a \emph{marking} if, for every $s\in S$, we can find an open neighborhood $U\subset S$ of $s$ with a homeomorphism $h: X \vert_U \lgr X_s \times U$ such that the restriction $F_u$ of the composition
\[ F: X \vert_U \lgr X_s \times U \xra{f_s \times \id_U} \Sigma \times U\]
to every point $u \in U$ satisfies
\[ [F_u] = [f_u]. \] 
For a morphism $\alpha: T \to S$, the pullback of a marking $\{[f_s]\}$ on $X$ to a marking on $X_T$ is defined by 
\[ \alpha^* \{ [f_s] \}= \{ [f_{\alpha(t)} ] \}_{t \in T}.\]

\begin{definition} A \emph{family of marked super Riemann surfaces} over an analytic superspace $S$ is a family $X\to S$ together with a marking $\{[f_s]\}_{s\in S}$.
An isomorphism
\[
(X/S,\{[f_s]\})\lgr(X'/S,\{[f_s']\})
\]
is a superconformal isomorphism $\Phi:X\lgr X'$ over $S$ satisfying
\[
[f_s]=[f_s'\circ \Phi_s]
\qquad
\text{for every }s\in S.
\]
\end{definition}
We write $\tg$ for the category fibered in groupoids over analytic superspaces whose objects over $S$ are families of marked super Riemann surfaces over $S$, and whose morphisms over $S$ are the isomorphisms just defined.
\medskip

A marked super Riemann surface over a point has exactly one non-trivial automorphism, called the \emph{fermionic reflection}.
To see this, recall that super Riemann surfaces over a point correspond to spin curves. The same equivalence identifies families over a purely bosonic base with families of spin curves. A spin curve is a triple $(C,L,\alpha)$ consisting of a Riemann surface $C$, a holomorphic line bundle $L$, and an isomorphism
\[
\alpha:L^{\otimes2}\lgr K_C.
\]
A marking on $(C,L,\alpha)$ is a marking $[f]$ on $C$.

An automorphism of a marked spin curve is a pair $(\phi,\theta)$, where $\phi:C\lgr C$ satisfies $[f\circ\phi]=[f]$ and
$\theta:\phi^*L\lgr L$ satisfies
\[
\alpha\circ\theta^{\otimes2}
=d\phi^*\circ\phi^*\alpha.
\]
Here $d\phi^*:\phi^*K_C\lgr K_C$ is the isomorphism induced by the differential.

Compatibility with the marking forces $\phi=\id_C$. The remaining condition is therefore
$\alpha\circ\theta^{\otimes2}=\alpha$. Since $\theta=a\,\id_L$ for some $a\in\bb^\times$, this condition gives $a^2=1$. Thus
\[
\theta\in\{\id_L,-\id_L\}.
\]
It follows that the automorphism group of a marked spin curve, and hence of a marked super Riemann surface over a point, is $\z/2\z$.


 \subsubsection{Marked Kuranishi families}

 A \emph{deformation of a marked super Riemann surface $(X,[f])$} over a pointed analytic superspace $(S, s_0)$ is a family of marked super Riemann surfaces $(\eta: Y \to S, \ [g_s] )$ with a fixed isomorphism $\psi: X \lgr Y\vert_{s_0}$ preserving both the superconformal structure and the marking: $[g_{s_0} \circ \psi]=[f]$. 
A \emph{marked Kuranishi family} for $(X, [f])$ is a marked deformation $(\pi: \mcx \to (B,b_0), \ [f_b], \  \phi: X \lgr \mcx_{b_0})$ over a $3g-3|2g-2$-dimensional pointed space $(B,b_0)$ with the analogous universal property: If $(\eta: Y \to (S,s_0), \ [g_s], \ \psi)$ is a marked deformation of $(X,[f])$ then there exists an open neighborhood $U \subset S$ of $s_0$ and an essentially unique cartesian pair $(h, H)$ where, in addition to the properties defined in the previous section, the morphism $H: Y|_U \to \mcx$ preserves the marking $[g_s]=[f_{h(s)} \circ H_s]$ for all $s \in U$. 
\medskip

Marked Kuranishi families always exist. Indeed, for any marked super Riemann surface $(X,[f])$, the marking $[f]$ can be extended uniquely to a marking $ [f_b]$ on the Kuranishi family \eqref{kuranishi} for $X$ such that $[f_{b_0} \circ \phi]=[f]$, and one then checks that the family
\begin{equation} \label{marked kuranishi}
    (\pi: \mcx \to (B,b_0), \  \ [f_b], \ \  \phi: X \lgr \mcx_{b_0})
\end{equation}
is a marked Kuranishi family for $(X,[f])$.

\subsection{$\tg$ is an analytic Deligne-Mumford superstack}

Ordinary Teichm\"uller space is a complex manifold. We cannot expect the same to hold for $\tg$ since  objects in $\tg$ have non-trivial automorphisms. Nevertheless,  $\tg$ has relatively simple geometry, namely, it is a Deligne-Mumford superstack with an explicit \'etale atlas given by a disjoint union of bases of marked Kuranishi families 
\begin{equation}
    \label{etale atlas for tg}
    \beta: \mc{B} \to \tg.
\end{equation}
where $\mc{B}$ is the disjoint union of the Kuranishi bases and $\beta$ is the disjoint union of the corresponding classifying maps. 

We first prove that all maps from complex analytic superspaces into $\tg$ are representable. Proving that $\tg$ is Deligne-Mumford then reduces to proving that the individual classifying maps in  \eqref{etale atlas for tg} are \'etale. 
 
\begin{lemma} \label{lemma :dig is rep}
    The diagonal morphism $\Delta: \tg \to \tg \times \tg$ is representable. 
\end{lemma}

\begin{proof}
The diagonal morphism is representable if for every map $(f,g): S \to \tg \times \tg$ with corresponding marked families denoted $(X\to S, [f_s])$ and $(Y\to S, [g_s])$, the marked isomorphism functor
\[
I':=\underline{\operatorname{Isom}}_{\tg(S)}(X,Y): \sS/S \to \on{Set}
\]
is representable by a complex analytic superspace. 

It is known that the superconformal isomorphism functor 
\[
I:=\underline{\operatorname{Isom}}_{\mfr(S)}(X,Y): \sS/S \to \on{Set}
\]
is a
finite and unramified complex analytic superspace over $S$ \cite[Proposition~4.3]{CV}, and so it suffices to prove that the obvious inclusion $I' \subset I$ is representable by open immersions. That is, for an $S$-space $j: T \to S$ and every map $T \to I$, the base change map $p: I_T':=I' \times_I T \to T$ is an open immersion.

Let $\alpha: X_T \lgr Y_T$ be the superconformal isomorphism corresponding to $T \to I$, and define
\[
U:=\{t\in|T|:[g_{j(t)} \circ \alpha_{j(t)}]=[f_{j(t)}]\}.
\]
Note that $U$ is open: Indeed, for every $u_0 \in U$ the fact that the two markings $[g_{j(u_0)} \circ \alpha_{j(u_0)}]$ and $[f_{j(u_0)}]$ agree implies that they agree on some  
 path-connected open neighborhood of $u_0$, and this open neighborhood is contained in $U$ by definition. 

Let $T'$ be an analytic superspace admitting a map $\beta': T' \to I_T'$ and set $\beta= p \circ \beta': T' \to T$.  Then by definition of $I_T'$, $\beta^* \alpha: X_{T'} \cong Y_{T'}$ is marking-preserving, so 
\[ [g_{(j \circ \beta)(t')} \circ \alpha_{(j \circ \beta)(t')}]= [f_{(j\circ \beta)(t')}]. \]
By definition of $U$, $\beta(t') \in U$, and thus $\beta$ factors through $U$. 
Conversely, suppose that a map $\beta: T' \to T$ factors through $U$, then $\beta^* \alpha$ satisfies the above equality which is exactly the marking-preserving condition, and so $\beta$ lifts to a map $\beta': T' \to I'$. 

\end{proof}


\begin{prop} \label{super teich is DM}  The map $B \to \tg$ corresponding to the family \eqref{marked kuranishi} 
is representable and \'etale. Furthermore, taking the disjoint union over all maps from all Kuranishi bases of marked super Riemann surfaces over a point gives an \'etale atlas
\begin{equation} \label{etale atlas super Teich} \alpha:  \mc{B} \to \tg.  \end{equation}

\end{prop}

\begin{proof} Let $S \to \tg$ be a morphism from a complex analytic superspace $S$ and let 
\[ (Y \xra{\eta} S, [g_s])\]
be the corresponding family of marked super Riemann surfaces. Set $B_S:= B \times_{\tg} S$, and let $p:B_S \to S$ and $q: B_S \to B$ denote the projection maps. 

The map $p$ is representable by \ref{lemma :dig is rep}, so it remains to show that it is \'etale. For every point $e \in B_S$, we will show there exists an open neighborhood $W$ of $B_S$ containing $e$ such that the restriction $p \vert_W: W \to S$ is an open immersion. 

By definition of the fiber product, a point $e \in B_S$ is a triple 
\[ (b=q(e), \ s=p(e), \ \sigma: \mcx_b \lgr Y_s)\]
where $\sigma$ is an isomorphism of marked super Riemann surfaces, so $[g_s \circ \sigma]=[f_b]$. 
Adding the isomorphism $\sigma$ to the data $(\eta: Y \to S, [g_s])$ makes the resulting triple a deformation of $\mcx_b$. The Kuranishi family of $(\mcx_b, [f_b])$ is $(\mcx \to (B,b), \id_{\mcx_b})$, and so there exists an open neighborhood $U \ni s$ in $S$ and a cartesian morphism
\[ \begin{tikzcd}
    Y|U \arrow[r, "F"] \arrow[d, "\eta_U"] & \mcx \arrow[d, "\pi"] \\
    U \arrow[r, "f"] & B. 
 \end{tikzcd}\]
 Let $\theta: \mcx \times_{B,f} U \lgr Y \vert_U$ be the induced isomorphism of marked families. 

 The triple $(i_U: U \subset S, f:U \to B, \theta)$ is an object in $B_S(U)$, and we write 
 \[ \psi:=(i_U, f): U \to B_S\]
for the corresponding map. This map factors uniquely into the composition
\[ U \xra{\psi_U} p\inv(U)=B \times_{\tg} U \xra{p_1} U, \qquad p_1 \circ \psi_U = \id_U.\]
We will now show that there exist open sets $W \subset p \inv(U)$ containing $e$ and $U' \subset U$ containing $s$, such that $p_1 \vert_W: W \to U$ factors as $ W \overset{p_W}{\cong} U' \subset U$. Note that this implies that $p\vert_W=p_W$ is an open immersion. 

Let $q_1$ denote the composition $p \inv(U) \subset B_S \xra{q} B$. By definition of the fiber product applied to $p \inv(U)= B \times_{\tg} U$, the identity morphism $\id_{ p \inv(U)}$ corresponds to a universal object 
\begin{equation} \label{universal triple} (p_1: p \inv(U) \to U, \ q_1: p \inv(U) \to B, \ \theta': \qonepull \mcx \lgr \ponepull (Y \vert_U)).\end{equation}
The triple corresponding to the map $\psi_U: U \to p \inv(U)$ is
\[ (\id_U, f:U \to B, \theta: \fpull \mcx \lgr Y \vert_U) \in p \inv(U)(U)\]
where $\psi_U \pull(\theta')=\theta$, and the triple corresponding to  $\psi_U \circ p_1: p \inv(U) \to p \inv( U)$ is
 \begin{equation} \label{comparison triple} (p_1,  \ f \circ p_1: p \inv(U) \to B, \  \ponepull \theta=(\psi_U \circ p_1) \pull \theta': (f \circ p_1)\pull \mcx \lgr \ponepull (Y \vert_U)).\end{equation}

 The first step is to show that there exists $W \subset p \inv(U)$ containing $e$ such that $\id_{p \inv(U)} \vert_W= \id_W$ is equal to $(\psi_U \circ p_1) \vert_W$. Proving this is equivalent to proving that there exists such a $W$ along which \eqref{universal triple} and \eqref{comparison triple}  agree. 

 Since $e \in p \inv(U)$, $\theta_e': \mcx_b \cong Y_s$ is equal to $\sigma=\ponepull(\theta)_e$. Let
 \[ \gamma':= \pr_{\mcx} \circ \theta'\inv, \qquad \gamma:=\pr_{\mcx} \circ (\ponepull \theta)\inv\]
 denote the composites of the horizontal maps in the following diagrams: 
 \[ \begin{tikzcd}
     (\ponepull Y \vert_U) \arrow[r, "\theta'\inv"] \arrow[dr] & \qonepull \mcx \arrow[d] \arrow[r, "\pr_{\mcx}"] & \mcx \arrow[d] \\
     {} & p \inv(U) \arrow[r,"q_1"] & B
 \end{tikzcd} \quad \begin{tikzcd}
     (\ponepull Y \vert_U) \arrow[r, "(\ponepull \theta)\inv"] \arrow[dr] & (f \circ p_1)\pull \mcx \arrow[d] \arrow[r, "\pr_{\mcx}"] & \mcx \arrow[d] \\
     {} & p \inv(U) \arrow[r,"f \circ p_1"] & B
 \end{tikzcd} \]
 where $\pr_{\mcx}$ in both diagrams just denotes the canonical projection. 
 The pairs 
 \[(q_1, \gamma'), \qquad (f \circ p_1, \gamma)\]
are morphisms of marked deformations of $\mcx_b$: 
\[ (\ponepull(Y \vert_U) \xra{\eta} (p \inv(U), e), \sigma) \to (\mcx \to (B,b), \id_{\mcx_b}).\]
Since the target is a Kuranishi family for $(\mcx_b,[f_b])$, and maps from a deformation to a Kuranishi family are essentially unique, there exists an open neighborhood $W \subset p \inv(U)$ containing $e$ such that
\[ (f \circ p_1) \vert_W=q_1 \vert_W, \qquad (\gamma')\vert_{\eta \inv(W)} =(\gamma) \vert_{\eta \inv(W)}.\]

We are left to show that $\theta'$ and $\ponepull \theta$ agree after restriction to $\eta \inv(W)$, or more precisely that their inverses agree after restriction to $\eta \inv(W)$. 
Since $q_1 \vert_W= (f \circ p_1) \vert_W$, we can write
\[ \mcx \times_B W:= \mcx \times_{B, q_1\vert_W} W = \mcx \times_{B, (f \circ p_1) \vert_W} W  \]
and consider the diagram 
\[ 
\begin{tikzcd}
\eta^{-1}(W) \arrow[r] \arrow[r, shift left] \arrow[rd, "\eta \vert_W"] & \mathcal{X} \times_B W \arrow[r, "\operatorname{pr}_{\mathcal{X}}"] \arrow[d, "\pi_W"] & \mathcal{X} \arrow[d, "\pi"] \\
                                                                        & W \arrow[r]                                                                            & B                           
\end{tikzcd}\]
where the two arrows from $\eta \inv(W)$ to $\mcx \times_B W$
are $\theta'\inv \vert_W$ and $(\ponepull \theta)\inv \vert_W$. Note that the two maps are equal if their composites with $\pr_{\mcx}$ and $\pi_W$ are equal. This holds because the composites of these maps with $\pr_{\mcx}$ are exactly $\gamma' \vert_{\eta \inv(W)}$ and $\gamma \vert_{\eta \inv(W)}$. 
Thus, we have shown that \eqref{universal triple} and \eqref{comparison triple} agree after restriction to $W$ and hence $\id_W= (\psi_U \circ p_1) \vert_W$.

Set $U' := \psi_U \inv(W) \subset U$, and note that $p_1 \vert_W: W \to U$ factors through $U'$. Denote this map by $p_W: W \to U'$. Furthermore, $\psi_U \vert_{U'}: U' \to p \inv(U)$ factors through $W$. Denote the resulting map by $\tilde{\psi}_U: U' \to W$, and note that we have a commutative diagram of restriction maps
\[ \begin{tikzcd} W \arrow[d, "i_W"] \arrow[r,"p_W"]& U' \arrow[d,"i_{U'}"] \arrow[r,"\tilde{\psi}_U"] & W \arrow[d, "i_W"] \arrow[r,"p_W"] & U' \arrow[d,"i_{U'}"]  \\
p \inv(U) \arrow[r,"p_1"] & U \arrow[r, "\psi_U"] & p \inv(U) \arrow[r,"p_1"] & U. 
\end{tikzcd} \]
where the subscripted symbols $i$ denote the canonical inclusions of the corresponding open subsets. 

By commutativity of the diagram and $\psi_U \circ p_1 \circ i_W= (\psi_U \circ p_1) \vert_W= \id_W$, we have that
\[ i_W \circ (\tilde{\psi_U} \circ p_W)=\id_W \ \Longrightarrow \tilde{\psi_U} \circ p_W=\id_W\]
since $i_W$ is an open immersion and thus monic. Similarly, 
\[ i_{U'} \circ (p_W \circ \tilde{\psi}_U) = p_1 \circ \psi_U \circ i_{U'}=i_{U'} \Longrightarrow p_W \circ \tilde{\psi}_U= \id_{U'}.\]
Thus, $p_W: W \to U'$ is an isomorphism, and thus $p \vert_W$ factors as $W \overset{p_W}{\cong} U' \subset U$, so it is an open immersion.

The statement that taking the disjoint union over all bases is an \'etale atlas is clear: The map is surjective as we have a marked Kuranishi family for every point in $\tg$, and each individual map in the disjoint union is \'etale. 
\end{proof}

\section{A Teichm\"uller theorem for spin and super Riemann surfaces}
\label{sec:teichmuller-theorem}

We first prove a version of Teichm\"uller's theorem for marked spin curves. This gives an analytic atlas whose underlying space is a disjoint union of open balls. We then lift this atlas to $\tg$. Each connected component of the resulting atlas has contractible underlying topological space.

Teichm\"uller's theorem \cite[Theorem~6.5]{AC} gives a homeomorphism
\begin{equation}\label{teich theorem}
\Phi:\ball:=\{q\in H^0(C,K_C^2)\mid \|q\|_1<1\}\longrightarrow T_g'
\end{equation}
from the unit ball to the moduli space $T_g'$ of marked continuous families of Riemann surfaces. The space $T_g'$ is canonically isomorphic to the topological space $|T_g|$ underlying $T_g$, so one usually states the theorem with $T_g$ as the target.

Throughout, we use $\ctop \to\ball$ to denote the continuous marked family whose classifying map is the homeomorphism in \eqref{teich theorem}. 

\subsection{Continuous families of marked spin curves}

A family of spin curves over a complex analytic space $A$ consists of a family of curves $\pi: C \to A$, a line bundle $\mcl$ on $C$, and an isomorphism
\[
\alpha:\mcl^{\otimes2}\lgr K_{C/A}.
\]
A marking is a marking on the underlying family of curves $C\to A$. An isomorphism of marked spin families consists of a marking-preserving isomorphism of spin curves. We write $ST_g$ for the analytic stack of spin curves. 

\begin{lemma}\label{lem:spin-bosonic}
There is a natural equivalence $(\tg)_{\bos}\simeq ST_g$. In particular, the bosonic reductions of the Kuranishi charts for $\tg$ form an \'etale atlas for $ST_g$.
\end{lemma}

We now work over the category of topological manifolds.  A \emph{continuous family of marked spin curves} is a continuous family of marked Riemann surfaces $\pi: C\to A$ in the sense of \cite[Section~5]{AC}, a continuous complex line bundle $\mcl$ with fiberwise holomorphic transition functions, and a continuous, fiberwise holomorphic isomorphism $\alpha:\mcl^{\otimes2}\lgr K_{C/A}$.


An isomorphism is a marking-preserving homeomorphism over $A$ that is biholomorphic on each fiber, together with an isomorphism of the spin line bundles preserving $\alpha$. We denote the stack of these families by $ST_g'$. 
\medskip

Let $C:= C_0^{top}$ denote the central fiber of the continuous family of Riemann surfaces $\ctop \to \ball$ picked out by the Teichm\"uller homeomorphism. By definition, $C$ is a Riemann surface, so we can equip it with a spin structure $(L_0, \alpha_0)$. 

The next lemma collects statements about extending spin structures and isomorphisms thereof from a fiber to a neighborhood of the base point of that fiber. These statements are well known, but we include them for future reference. 

\begin{lemma}[Extension Lemma] \label{the extension lemmas}
Let $\pi:Y\to S$ be a continuous family of Riemann surfaces over
a topological manifold $S$, and fix a point $0\in S$.
For an open neighbourhood $U\subset S$ of $0$, write
$Y_U:=\pi^{-1}(U)$.

\begin{enumerate}
    \item \emph{Extension of objects.}
    For every spin structure $(L_0,\alpha_0)$ on $Y_0$, we can find an open neighborhood $U\subset S$ of $0$, a
    continuous spin structure $(L,\alpha)$ on $Y_U$, and an identification
    \[
        \iota:(L,\alpha)|_{Y_0}
        \xrightarrow{\sim}(L_0,\alpha_0).
    \] 
    This extension is essentially unique in the sense that any other extension agrees with it after restriction to some common neighborhood. 

    \item \emph{Extension of isomorphisms.}
    Let $(L,\alpha_L)$ and $(M,\alpha_M)$ be continuous spin structures
    on $Y$ and let
    \[
        \gamma_0:M|_{Y_0}\xrightarrow{\sim}L|_{Y_0}
    \]
    be an isomorphism of spin curves. Then we can find an open neighborhood $U\subset S$ of $0$, and a
    continuous isomorphism of families of spin curves
    \[
        \gamma:M|_{Y_U}\xrightarrow{\sim}L|_{Y_U}.
    \]

    The extension $\gamma$ is essentially unique. Furthermore, the same holds when $\pi:Y\to S$ is an analytic
    family of Riemann surfaces over a complex analytic space $S$,
    and the two spin structures are analytic. The extension $\gamma$
    is then analytic.

    \item \emph{Extension of complex structure.}
    Let $X$ be a complex manifold, let $N$ be a holomorphic line
    bundle on $X$, and let $L$ be a continuous complex line bundle
    on $X$, equipped with an isomorphism of continuous line bundles
    \[
        \alpha:L^{\otimes 2}\xrightarrow{\sim}N.
    \]
    Then $L$ admits a unique holomorphic structure for which
    $\alpha$ is holomorphic.

    In particular, if $\pi:Y\to S$ is an analytic family of Riemann
    surfaces with $Y$ and $S$ complex manifolds, every continuous
    spin structure
    \[
        \alpha:L^{\otimes 2}\xrightarrow{\sim}\omega_{Y/S}
    \]
    acquires a unique analytic spin structure compatible with
    $\alpha$. Its restriction to each fiber agrees with the
    original holomorphic spin structure on that fiber.
\end{enumerate}

For two marked curve families $\pi_K:K\to S$ and $\pi_Y:Y\to S$,
the corresponding isomorphism-extension statement applies after
fixing an isomorphism of underlying marked curve families
\[
    F:K\xrightarrow{\sim}Y,
    \qquad
    \pi_Y\circ F=\pi_K.
\]
One applies statement~(2) on $K$ to the spin structures $L$ and $F^*M$.
In the analytic setting, $F$ is required to be analytic, and one
applies statement~(3).
\end{lemma}

\begin{proof}
We prove part (1). The proofs of the remaining statements are similar in spirit. 
Fix a homeomorphism $h:\ctop \lgr C \times\ball$ over
$\ball$, with $h(c)=(c,0)$ on the central fiber, and define
$L:=h^*(L_\sigma\times\ball)$.
Then
\[
h^*(L_\sigma^{\otimes2}\times \ball)
\cong
\bigl(h^*(L_\sigma\times\ball)\bigr)^{\otimes2}
=
L^{\otimes2}
\]
and pulling back
$\alpha_\sigma\times\id_{\ball}:
L_\sigma^{\otimes2}\times \ball\lgr K_C\times \ball$
gives an isomorphism
\[
\widetilde\alpha:
L^{\otimes2}\lgr h^*(K_C\times \ball).
\]

The target of this map is isomorphic to $K_{\ctop/\ball}$. Indeed, both restrict to
$K_C$ on the central fiber and, since $\ball$ is contractible,
restriction to that fiber determines the isomorphism class of any
topological complex line bundle on $\ctop$. 
We can therefore choose an isomorphism
$\beta_0:h^*(K_C\times \ball)\lgr K_{\ctop/\ball}$ whose restriction to the central fiber is the identity map. The pair $(L, \beta_0 \circ \tilde{\alpha})$ is then a continuous spin structure on $\ctop$ whose restriction to the central fiber is $(L_{\sigma}, \alpha_{\sigma})$
\end{proof}


Let $\on{Spin}(C)$ be the set of isomorphism classes of spin structures on $C$. We know that the set $\on{Spin}(C)$ has $2^{2g}$ elements and is a torsor for two-torsion line bundles on $C$. Choose a representative $(L_\sigma,\alpha_\sigma)$ for each $\sigma\in\on{Spin}(C)$.

Use one copy $\ball$ of the ball for each spin structure, and write $K_{\sigma} \to \ballsig$ for the continuous family of marked spin curves whose underlying continuous family is $C \to \ball$ and whose spin structure $(L, \alpha)$ is the one extending $(L_{\sigma}, \alpha_{\sigma})$ on the central fiber. We can take a disjoint union of these as we vary over $\sigma \in \on{Spin}(C)$ and define the following continuous family of marked spin curves: 
\begin{equation} \label{cont fams of marked spins} K:=\bigsqcup_{\sigma} K_{\sigma} \to  B:= \bigsqcup_{\sigma} \ballsig\end{equation}

\begin{theorem}\label{teich theorem for spin} The classifying map 
 $B \to ST_g'$ of the family in \eqref{cont fams of marked spins} is surjective, representable and \'etale.  
\end{theorem}

\begin{proof}
We first prove that $B \to ST_g'$ is representable. The arguments used to prove representability also show that the map is \'etale. 
It suffices to prove that the component maps $\ballsig \to  ST_g'$ are representable. 
\medskip 

For $S \to ST_g'$ let $(Y \to S, M)$ denote the corresponding continuous family of marked spin curves. Set $E:= \ballsig \times_{ST_g'} S$, and let $p:E \to S$ and $q: E \to \ballsig$ denote the natural projection maps. 

Fix a point $e \in E$, and let $(b=q(e), \ s= p(e), \ \theta_e: (\ksig)_b \lgr Y_s)$ denote the corresponding object, where $\theta_e$ is an isomorphism of marked spin curves. 

We claim that there exist an open neighborhood $U \subset S$ containing $s$, a map $f:U \to \ballsig$, and an isomorphism
\[ F: \fpull K_{\sigma} \lgr Y \vert_U\]
of continuous families of marked spin curves over $U$ such that $F_s= \theta_e$. 

Let $c$ denote the composition $S \to ST_g' \to T_g'=|T_g|$ and let $\Phi_{\sigma}: \ballsig \to T_g'=|T_g|$ be the Teichm\"uller homeomorphism, and let $C_{\sigma} \to \ballsig $ be the corresponding continuous family of Riemann surfaces. Set
\[ f:= \Phisig \inv \circ c\]
and note that $f(s)=b$ since $c(s)=\Phisig(b)$ as elements in $T_g'$. Indeed, these correspond to $(C_{\sigma})_b$ (the continuous curves underlying $(\ksig)_b$) and $Y_s$ respectively, and we know these are isomorphic via $\theta_e$. 
Indeed, $\theta_e$ is an isomorphism of spins so in particular one of the underlying curves. 
By the same logic, for every $t \in S$, the identity $c(t)=\Phisig(f(t))$ implies that there exists an isomorphism
\[ F_t: (C_{\sigma})_{f(t)} \lgr Y_t \]
of marked Riemann surfaces, and since these isomorphisms are unique, they glue to a global isomorphism of continuous families of marked Riemann surfaces
\[ F: \fpull C_{\sigma} \lgr Y, \qquad F \vert_s = F_s\]
over $S$ by \cite[Theorem~5.6]{AC}. By Lemma~\ref{the extension lemmas}, there exists an open neighborhood $U \subset S$ of $s$ and an isomorphism of continuous families of marked spin curves
\[ \Theta: (f|_U)^* \ksig \lgr Y|U, \qquad \Theta_s= \theta_e\]
and the underlying continuous isomorphism is $F|_U$. Redefine $f:= f |_U:U \to \ballsig$, and let 
\[\psi: U \to E\]
denote the classifying map of the object $(i_U:U \subset S, \ f: U \to \ballsig, \ \Theta)$.

We will now prove that $\psi$ is representable by an open immersion. That is, for every $\tau: T \to E$ we will show that there exists an open subset $W \subset T$ which represents $U \times_E T$. Let $a$ denote the composition $T \xra{\tau} E \xra{p} S$, and define
\[ W:= \{ t \in a\inv(U) \mid \tau(t) = \psi(a_U(t)) \}  \subset T \times_S U\]
where $a_U: T \times_S U \to U$ is the projection map. 

We first show that $W$ is open. Fix $t_0 \in W$, and let 
\[ G: (q \circ \tau) \ksig \lgr (p \circ \tau)Y\]
be the isomorphism picked out by $\tau$, and 
\[ a_U^* \Theta: a_U^* f^* \ksig \lgr a_U^*(Y \vert_U)\]
the pullback of the isomorphism $\Theta$ picked out by $\psi$. 
Then by definition of $W$, $G_{t_0} = (a_U^*\Theta)_{t_0}$, and by uniqueness of extension of isomorphisms of marked spin curves on the fibers, there must exist some open subset $N$ containing $t_0$ such that $G \vert_N = (a_U^*\Theta) \vert_N$, and hence $N \subset W$. Thus, $W$ is open, and in particular open in $T$. 

We now verify that $W$ represents $U \times_E T$. Indeed, a map $h: R \to T$ lifts to $R \to U \times_E T$ if and only if $R$ admits a map $u: R \to U$ satisfying $\tau \circ h= \psi \circ u$. The definition of $W$ and the fiber product $E=S \times_{ST_g'} \ballsig$ then immediately imply that the map $h$ must factor through $W$. 

The next step is to prove that $E$ is representable by a topological manifold. Note that once we have established this, we have also proved that the map $E \to S$ is \'etale. Indeed, we have shown that for every $e \in E$, there exists an open subset $V:=\psi(U) \subset E$ which $p$ maps isomorphically onto the open subset $U \subset S$. 
We will use the following descent statement from the open-gluing and Hausdorff criteria in \cite[Theorems~1.1 and~1.3]{conradGluing}: A sheaf $E$ covered by open immersions $\psi_e: U_e \to E$ is represented by a topological manifold if it admits a countable subcover and each canonical map
\[ \rho:=(\pr_1, \pr_2): R \to U_e \cap U_{e'}=U_e \times_S U_{e'}, \]
where $R:=U_e \times_{\psi_e, E, \psi_{e'}} U_{e'}$ denotes the fiber product in the diagram 
\[ \begin{tikzcd}
    R \arrow[r, "\pr_1"] \arrow[d, "\pr_2"] & U_e \arrow[d, "\psi_e"] \\
    U_{e'} \arrow[r,"\psi_{e'}"] & E,
\end{tikzcd}\]
is a closed immersion. 

First note that both $\pr_1$ and $\pr_2$ are open immersions, since they are base changes of the open immersions $\psi_{e'}$ and $\psi_e$ respectively. Set $O:= \pr_2(R) \subset U_{e'}$, and let $W_{ee'}:= j\inv(O)$, where
$j: U_e \cap U_{e'} \to U_{e'}$ is the inclusion map. The map
\[ j \vert_{W_{ee'}}: W_{ee'} \to O\]
is a homeomorphism, and hence  $\rho: R \to W_{ee'}$ is a homeomorphism. 

We are left to show that $W_{ee'}$ is closed. Let 
\[ f_e: U_e \to \ballsig, \qquad f_{e'}: U_{e'} \to \ball_{\sigma'}\] 
and note that $U_e \cap U_{e'} = \emptyset$ if $\sigma \neq \sigma'$, so we may assume that $\sigma=\sigma'$. Then clearly 
\[ f_e \vert_{U_e \cap U_{e'}} = f_{e'} \vert_{U_e \cap U_{e'}}\]
since  $f_e:= \Phisig \inv \circ c \circ i_{U_e}$ and $f_{e'}:= \Phisig \inv \circ c \circ i_{U_{e'}}$. Let us denote the common map by
\[f: U_e \cap U_{e'} \to \ballsig,\]
and let $Z= f^* C_{\sigma} \to U_e \cap U_{e'}$ be the underlying continuous family of marked curves.  
Recall the isomorphisms of continuous families of marked Riemann surfaces underlying $\Theta_e$ and $\Theta_{e'}$: 
\[ F_e: f_e^* C_{\sigma} \to Y, \qquad F_{e'}: f_{e'}^* C_{\sigma} \to Y. \]
Since isomorphisms of marked Riemann surfaces are unique, $F_e \vert_{U_e \cap U_{e'}}= F_{e'} \vert_{U_e \cap U_{e'}}: Z \lgr Y \vert_{U_e \cap U_{e'}}$.   

Write $V:=U_e\cap U_{e'}$, let $\pi:Z\to V$ be the projection, and denote the common underlying isomorphism by
\[
    F:Z\lgr Y_V, \qquad Y_V:=Y\vert_V.
\]
Let $(L,\alpha_L)$ denote the spin structure on $Z$ belonging to $f^*K_{\sigma}$, and write $M_V:=M\vert_{Y_V}$. The restrictions of $\Theta_e$ and $\Theta_{e'}$ to $V$ determine isomorphisms of line bundles on $Z$
\[
    \gamma_e,\gamma_{e'}:F^*M_V\lgr L
\]
compatible with the spin structures. Their ratio
\[
    h:=\gamma_{e'}\circ\gamma_e^{-1}:L\lgr L
\]
is therefore an automorphism of the spin structure $(L,\alpha_L)$. Since $L$ is a line bundle, there is a unique continuous, fiberwise holomorphic function $u:Z\to\mathbb{C}^{\times}$ such that $h=u\,\operatorname{id}_L$. Compatibility with $\alpha_L$ gives $u^2=1$, so $u$ is locally constant with values $1$ and $-1$.

Define the open subset
\[
    D:=\{z\in Z:u(z)\neq 1\}\subset Z.
\]
We claim that
\[
    V\setminus W_{ee'}=\pi(D).
\]
Indeed, for $s\in V$, the objects $\psi_e(s)$ and $\psi_{e'}(s)$ have the same base point $s$, the same point $f(s)\in\ballsig$, and the same underlying marked-curve isomorphism $F_s$. Thus, writing $Z_s:=\pi^{-1}(s)$, the definition of the fiber product gives
\[
\begin{aligned}
    s\in W_{ee'}
    &\Longleftrightarrow \psi_e(s)=\psi_{e'}(s)\\
    &\Longleftrightarrow \gamma_e\vert_{Z_s}=\gamma_{e'}\vert_{Z_s}\\
    &\Longleftrightarrow u\vert_{Z_s}\equiv 1.
\end{aligned}
\]
Consequently, $s\notin W_{ee'}$ if and only if there exists $z\in Z_s$ with $u(z)\neq 1$, which is precisely the condition $s\in\pi(D)$. Since $Z_s$ is connected, in this case $u\vert_{Z_s}\equiv -1$.

Finally, the projection $\pi:Z\to V$ is open: the local product charts in the definition of a continuous family identify its restrictions with projections from products, which are open maps. Hence $\pi(D)$ is open in $V$. The displayed equality shows that $W_{ee'}$ is closed in $V=U_e\cap U_{e'}$. Under the identification $R\cong W_{ee'}$ established above, the map $\rho:R\to U_e\cap U_{e'}$ is therefore a closed immersion, as required.

\end{proof}

\subsection{The analytic spin family and the lift to super Teichm\"uller space}

Transport the complex structure of $T_g$ along $\Phi_\sigma:\ballsig\lgr |T_g|$, and denote the resulting complex manifold by $\ballsig^{\on{an}}$. The same map is then a biholomorphism $\Phi_\sigma^{\on{an}}:\ballsig^{\on{an}}\lgr T_g$.
Let $\pi_{\sigma}: C_{\sigma} \to \ballsig$ denote the continuous family of marked curves corresponding to the Teichm\"uller homeomorphism $\Phisig$. 

Let $q_g:\cc_g\to T_g$ be the universal family of marked Riemann surfaces and let $q_\sigma:X_\sigma:= \cc_g \times_{T_g} \ballsig^{an} \to\ballsig^{\on{an}}$ be the base change along the biholomorphism $\Phi_\sigma^{\on{an}}$. The continuous family underlying $q_{\sigma}: X_\sigma \to \ballsig^{an}$ and $\pi_\sigma:C_\sigma\to\ballsig$ have the same classifying map $\Phi_\sigma$. By \cite[Theorem~5.6]{AC}, there is a unique isomorphism of continuous marked families
\[
    H_\sigma:C_\sigma\lgr X_\sigma,
    \qquad q_\sigma \circ H_\sigma=\pi_\sigma.
\]
Transporting the complex structure along $H_\sigma$, we make $\pi_\sigma: C_\sigma \to \ballsig$ into an analytic family of marked curves. Denote it by $\pi_\sigma^{\on{an}}:C_\sigma^{\on{an}}\to\ballsig^{\on{an}}$. The original complex structures on its fibers are compatible with the complex structure transported from $H_\sigma$. Furthermore, by Lemma~\ref{the extension lemmas}, the complex structure on $C_{\sigma}$ extends to a complex structure on its continuous spin structure $(L, \alpha)$. We write  $K_\sigma=(C_\sigma^{\on{an}},L,\alpha)\to\ballsig^{\on{an}}$ for the resulting analytic family of marked spin curves. Taking the disjoint union of these families and their classifying maps into $ST_g$ defines a map 
\begin{equation}\label{analytic spin atlas}
B^{\on{an}}:=\bigsqcup_\sigma\ballsig^{\on{an}}
    \longrightarrow ST_g.
\end{equation}


\begin{prop}\label{prop:analytic-spin-atlas}
The map \eqref{analytic spin atlas} is surjective, representable and \'etale.
\end{prop}
\begin{proof}
Representability follows from the representable diagonal of the analytic Deligne--Mumford stack $ST_g$. We prove \'etaleness componentwise. Given an analytic family $\mathcal{Y}=(Y\to S,M,\alpha_M)$ classified by $S\to ST_g$, set $E:=\ballsig^{\on{an}}\times_{ST_g}S$, with projections $p:E\to S$ and $q:E\to\ballsig^{\on{an}}$.

Fix $e=(s,b,\theta_e)\in E$, where $s=p(e)$, $b=q(e)$ and $\theta_e:(K_\sigma)_b\lgr\mathcal{Y}_s$. If $c:S\to T_g$ is the analytic classifying map of the underlying marked curve family, then
\[
    f:=(\Phi_\sigma^{\on{an}})^{-1}\circ c:S\to\ballsig^{\on{an}}
\]
is analytic and satisfies $f(s)=b$. We obtain an isomorphism $F:f^*C_\sigma^{\on{an}}\lgr Y$ of families of marked curves over $S$ whose restriction at $s$ is  $\theta_e$. By Lemma~\ref{the extension lemmas} we can find an open neighborhood $U\subset S$ of $s$, and an isomorphism
\[
    \Theta:(f\vert_U)^*K_\sigma\lgr\mathcal{Y}\vert_U,
    \qquad \Theta_s=\theta_e.
\]
of marked families of spin curves. 
The triple $(i_U,f\vert_U,\Theta)$ defines $\psi_e:U\to E$, with $i_U:U\hookrightarrow S$ the inclusion and $p\circ\psi_e=i_U$. The same open-immersion argument as in Theorem~\ref{teich theorem for spin} can then be used to show that each $\psi_e$ is an open immersion, and thus $E \to S$ is \'etale.  Finally, the surjectivity in Theorem~\ref{teich theorem for spin}, together with these local analytic lifts, proves surjectivity of \eqref{analytic spin atlas}.
\end{proof}

For a Deligne--Mumford superstack, the map $\mc{M}_{\bos}\to\mc{M}$ is a representable closed immersion.
\begin{theorem}\label{teich theorem for super}
The super Teichm\"uller stack $\tg$ admits an \'etale atlas
\[
\mc{P}\longrightarrow\tg
\]
whose bosonic reduction is \eqref{analytic spin atlas}. In particular, 
\[
|\mc{P}|\cong\bigsqcup_{\sigma\in\on{Spin}(C)}\ball_\sigma.
\]
Thus its underlying topological space is homeomorphic to a disjoint union of $2^{2g}$ contractible unit balls.
\end{theorem}
\begin{proof}
By Lemma~\ref{lem:spin-bosonic}, the bosonic reduction of $\tg$ is $ST_g$. By \'etale lifting, there exists a complex analytic superspace $\mc{P}$, and an \'etale map $\mc{P} \to \tg$ filling in the cartesian square
\[
\begin{tikzcd}
\pan \arrow[r] \arrow[d,"p^{\on{an}}"'] & \mc{P} \arrow[d] \\
ST_g \arrow[r] & \tg.
\end{tikzcd}
\]

\end{proof}

\section*{Acknowlegments} I want to thank Daniel Hernández Ruipérez for his helpful comments on the canonical locus in the Hilbert superscheme and the representability of the Quot superscheme.

\bibliographystyle{amsalpha}
\bibliography{References}

@article{ott2023supermoduli,
  title={The supermoduli space of genus zero super {Riemann} surfaces with {Ramond} punctures},
  author={Ott, Nadia and Voronov, Alexander A.},
  journal={Journal of Geometry and Physics},
  volume={185},
  pages={104726},
  year={2023},
  publisher={Elsevier}
}

@article{Grothendieck1961,
  author = {Grothendieck, Alexander},
  title = {Techniques de construction en g{\'e}om{\'e}trie analytique. {X}. Construction de l'espace de {Teichm{\"u}ller}},
  journal = {S{\'e}minaire Henri Cartan},
  volume = {13},
  number = {2},
  year = {1960--1961},
  pages = {1--20},
  note = {Expos{\'e} no.~17},
  url = {https://www.numdam.org/item/SHC_1960-1961__13_2_A4_0/}
}

@article{AC,
  author = {Arbarello, Enrico and Cornalba, Maurizio},
  title = {{Teichm{\"u}ller} space via {Kuranishi} families},
  journal = {Annali della Scuola Normale Superiore di Pisa, Classe di Scienze},
  series = {5},
  volume = {8},
  number = {1},
  year = {2009},
  pages = {89--116},
  url = {https://www.numdam.org/item/ASNSP_2009_5_8_1_89_0/}
}

@article{CR,
  author = {Crane, Louis and Rabin, Jeffrey M.},
  title = {Super {Riemann} surfaces: uniformization and {Teichm{\"u}ller} theory},
  journal = {Communications in Mathematical Physics},
  volume = {113},
  number = {4},
  year = {1988},
  pages = {601--623},
  doi = {10.1007/BF01223239}
}

@article{LR,
  author = {LeBrun, Claude and Rothstein, Mitchell},
  title = {Moduli of super {Riemann} surfaces},
  journal = {Communications in Mathematical Physics},
  volume = {117},
  number = {1},
  year = {1988},
  pages = {159--176},
  doi = {10.1007/BF01228415}
}

@article{CV,
  author = {Codogni, Giulio and Viviani, Filippo},
  title = {Moduli and periods of supersymmetric curves},
  journal = {Advances in Theoretical and Mathematical Physics},
  volume = {23},
  number = {2},
  year = {2019},
  pages = {345--402},
  doi = {10.4310/ATMP.2019.v23.n2.a2},
  eprint = {1706.04910},
  archivePrefix = {arXiv},
  primaryClass = {math.AG}
}

@article{PZ,
  author = {Penner, R. C. and Zeitlin, Anton M.},
  title = {Decorated super-{Teichm{\"u}ller} space},
  journal = {Journal of Differential Geometry},
  volume = {111},
  number = {3},
  year = {2019},
  pages = {527--566},
  doi = {10.4310/jdg/1552442609}
}

@article{rabin1987teichmuller,
  title={{Teichm{\"u}ller} deformations of super {Riemann} surfaces},
  author={Rabin, Jeffrey M.},
  journal={Physics Letters B},
  volume={190},
  number={1--2},
  pages={40--46},
  year={1987},
  publisher={Elsevier}
}

@article{goldman1988topological,
  title={Topological components of spaces of representations},
  author={Goldman, William M.},
  journal={Inventiones mathematicae},
  volume={93},
  number={3},
  pages={557--607},
  year={1988},
  publisher={Springer}
}

@book{teichmuller1940extremale,
  title={Extremale quasikonforme {Abbildungen} und quadratische {Differentiale}},
  author={Teichm{\"u}ller, Oswald},
  year={1940},
  publisher={Verlag der Akademie der Wissenschaften}
}

@article{sachse2009global,
  title={Global analytic approach to super {Teichm{\"u}ller} spaces},
  author={Sachse, Christoph},
  journal={arXiv preprint arXiv:0902.3289},
  year={2009}
}

@article{knevel2011versal,
  title={Versal Families of Compact Super {Riemann} Surfaces},
  author={Knevel, Roland},
  journal={arXiv preprint arXiv:1103.0664},
  year={2011}
}

@article{bruzzo2023notes,
  title={Notes on fundamental algebraic supergeometry. {Hilbert} and {Picard} superschemes},
  author={Bruzzo, Ugo and Ruip{\'e}rez, Daniel Hern{\'a}ndez and Polishchuk, Alexander},
  journal={Advances in Mathematics},
  volume={415},
  pages={108890},
  year={2023},
  publisher={Elsevier}
}

@article{ott2026algebraicity,
  title={Algebraicity of supermoduli of curves via {Artin}'s criteria},
  author={Ott, Nadia},
  journal={arXiv preprint arXiv:2603.16254},
  year={2026}
}

@misc{conradGluing,
  author = {Conrad, Brian},
  title = {Gluing topologies, the {Hausdorff} condition, and examples},
  howpublished = {Math 396 course handout, Stanford University},
  url = {https://math.stanford.edu/~conrad/diffgeomPage/handouts/gluing.pdf}
}

@unpublished{POP,
  author = {Ott, Nadia and Polishchuk, Alexander and Peng, Fei},
  title  = {Projectivity of coarse moduli of supercurves},
  note   = {Unpublished manuscript},
  year   = {2026}
}

@unpublished{ADT,
  author = {Alvarez-Cónsul, Luis and Ruip{\'e}rez, Daniel Hern{\'a}ndez  and Gómez, Tomás},
  title  = {The moduli superstack of semistable coherent sheaves on a projective superscheme},
  note   = {Unpublished manuscript},
  year   = {2026}
}

@article{bryant1993nielsen,
  author = {Bryant, P. and Hodgkin, L.},
  title = {{Nielsen}'s theorem and the super-{Teichm{\"u}ller} space},
  journal = {Annales de l'Institut Henri Poincar\'e. Physique Th\'eorique},
  volume = {58},
  number = {3},
  year = {1993},
  pages = {247--265},
  url = {https://www.numdam.org/item/AIHPA_1993__58_3_247_0/}
}











\end{document}